\documentclass[11pt]{article}

\usepackage{amssymb,amsmath,amsfonts}
\usepackage{graphicx,color,enumitem}
\usepackage{amsthm}	
\usepackage{bm}
\usepackage[round]{natbib}
\usepackage{epstopdf}

\usepackage{etoolbox}

\RequirePackage[colorlinks,citecolor=blue,urlcolor=blue]{hyperref}

\makeatletter

\patchcmd{\NAT@citex}
  {\@citea\NAT@hyper@{%
     \NAT@nmfmt{\NAT@nm}%
     \hyper@natlinkbreak{\NAT@aysep\NAT@spacechar}{\@citeb\@extra@b@citeb}%
     \NAT@date}}
  {\@citea\NAT@nmfmt{\NAT@nm}%
   \NAT@aysep\NAT@spacechar\NAT@hyper@{\NAT@date}}{}{}

\patchcmd{\NAT@citex}
  {\@citea\NAT@hyper@{%
     \NAT@nmfmt{\NAT@nm}%
     \hyper@natlinkbreak{\NAT@spacechar\NAT@@open\if*#1*\else#1\NAT@spacechar\fi}%
       {\@citeb\@extra@b@citeb}%
     \NAT@date}}
  {\@citea\NAT@nmfmt{\NAT@nm}%
   \NAT@spacechar\NAT@@open\if*#1*\else#1\NAT@spacechar\fi\NAT@hyper@{\NAT@date}}
  {}{}

\makeatother

\newcommand{\e}{{\rm e}}

\newcommand{\E}{{\mathbb E}}

\renewcommand{\P}{{\mathbb P}}

\newcommand{\R}{{\mathbb R}}

\newcommand{\N}{{\mathbb N}}
\newcommand{\T}{{\mathbf T}}

\newcommand{\Bcal}{{\mathcal B}}

\newcommand{\Fcal}{{\mathcal F}}
\newcommand{\Gcal}{{\mathcal G}}
\newcommand{\Hcal}{{\mathcal H}}

\newcommand{\Tcal}{{\mathcal T}}

\DeclareMathOperator{\esssup}{ess\,sup}
\DeclareMathOperator{\rk}{rank}

\DeclareMathOperator{\conv}{conv}

\newcommand{\Pol}{{\rm Pol}}
\newcommand{\Id}{{\mathrm{Id}}}

\DeclareMathOperator{\tr}{Tr}

\newtheorem{theorem}{Theorem}

\newtheorem{corollary}[theorem]{Corollary}

\newtheorem{definition}[theorem]{Definition}
\newtheorem{example}[theorem]{Example}

\newtheorem{lemma}[theorem]{Lemma}

\newtheorem{proposition}[theorem]{Proposition}
\newtheorem{remark}[theorem]{Remark}

\numberwithin{equation}{section}
\numberwithin{theorem}{section}

\begin{document}

\title{Markov cubature rules for polynomial processes\footnote{The authors would like to thank the anonymous referees for their careful reading of the manuscript and suggestions. Martin Larsson gratefully acknowledges support from SNF Grant 205121\_163425. The research of Sergio Pulido benefited from the support of the Chair Markets in Transition (F\'ed\'eration Bancaire Fran\c caise) and the project ANR 11-LABX-0019. The research leading to these results has received funding from the European Research Council under the European Union's Seventh Framework Programme (FP/2007-2013) / ERC Grant Agreement n. 307465-POLYTE}}
\author{Damir Filipovi\'c\footnote{EPFL and Swiss Finance Institute, 1015 Lausanne, Switzerland. {\it Email: }damir.filipovic@epfl.ch} \and Martin Larsson\footnote{ETH Zurich, Department of Mathematics, R\"amistrasse 101, CH-8092, Zurich, Switzerland. {\it Email: }martin.larsson@math.ethz.ch} \and Sergio Pulido\footnote{Laboratoire de Math\'ematiques et Mod\'elisation d'\'Evry (LaMME), Universit\'e d'\'Evry-Val-d'Essonne, ENSIIE, Universit\'e Paris-Saclay, UMR CNRS 8071, IBGBI 23 Boulevard de France, 91037 \'Evry Cedex, France. {\it Email: }sergio.pulidonino@ensiie.fr}}
\date{June 10, 2019\\forthcoming in {\it Stochastic Processes and their Applications}}

\maketitle

\begin{abstract}
We study discretizations of polynomial processes using finite state Markov processes satisfying suitable moment matching conditions. The states of these Markov processes together with their transition probabilities can be interpreted as Markov cubature rules. The polynomial property allows us to study such rules using algebraic techniques. Markov cubature rules aid the tractability of path-dependent tasks such as American option pricing in models where the underlying factors are polynomial processes.
\\[2ex] 
\noindent{\textbf {Keywords:} Polynomial process; cubature rule; asymptotic moments; transition rate matrix; transition probabilities; American options.}
\\[2ex]
\noindent{\textbf {MSC2010 subject classifications:} 60G07, 60J25, 60J27, 60J28, 60J10, 91G60, 65C20, 65C30, 60H35, 60F99, 60J60, 60J75, 60H10, 60H20, 60H30.}
\end{abstract}


\section{Introduction}
\label{sec:intro}
Polynomial processes have recently gained popularity thanks to their tractability and flexibility. For instance, they have been applied in financial market models for interest rates \citep{Delbaen/Shirakawa:2002,Zhou2003,Filipovic/Larsson/Trolle:2014}, credit risk \citep{Ackerer/Filipovic:2016}, variance swaps \citep{Filipovic/Gourier/Mancini:2015}, stochastic volatility \citep{Ackerer/Filipovic/Pulido:2016}, stochastic portfolio theory \citep{C:16}, life insurance liabilities \citep{BZ:16}, energy prices \citep{FilipovicLarsonWare:18}, and foreign exchange rates \citep{DeJongetal01,LarsenSorensen07}. Polynomial processes, as considered by\ \cite{Cuchiero/etal:2012} and~\cite{filipovic2016polynomial,FilipovicLarson:17}, are stochastic processes with the property that the conditional expectation of a polynomial is a polynomial of the same or lower degree. This implies that conditional moments can be computed efficiently and accurately, which can be exploited to construct tractable models. Despite these advantages, the tractability of polynomial processes deteriorates as one faces path-dependent tasks such as American option pricing or computation of path-dependent functionals.

In this paper, we develop a method for tackling such problems. We approximate a given polynomial process by a finite state Markov process that matches moments up to a given order. We call such a finite state process a \emph{Markov cubature rule} because the states of the process together with their transition probabilities can be interpreted as cubature rules for the law of the original process at different times. Markov cubature rules facilitate the implementation of polynomial models by simplifying costly computational tasks such as Monte-Carlo simulation and pricing of path-dependent and American options.

The polynomial property allows us to study the existence of Markov cubature rules using algebraic techniques. Contrary to the classical cubature problem, we look for cubature rules that use the same set of cubature points at all times, as this is desirable for numerical applications like the calculation of American option prices in finance. Additionally, the moments to be matched depend on the cubature points chosen. In continuous time, the exact moment matching condition turns out to be too stringent as we explain in Section~\ref{sec:MCandliftedMC}. Instead, we find {\em approximate} Markov cubature rules by solving a quadratic programming problem. This quadratic programming problem arises naturally from our first main result, Theorem~\ref{thm:continuouscubature}, which gives an algebraic and geometric characterization of continuous time Markov cubature rules. While a systematic analysis of computational cost, accuracy, and convergence falls outside the scope of the present paper, we provide numerical examples which indicate that the approximate Markov cubature rules work well in practice. In discrete time, our second main result, Theorem~\ref{thm:discretecubature}, yields existence of Markov cubature rules on an appropriately chosen time grid, under suitable assumptions involving the asymptotic moments of the given polynomial process. The existence of asymptotic moments is a crucial hypothesis and lies at the core of the proof of this theorem.

Approximations by discretization of stochastic models using finite state Markov processes appear regularly in the numerical methods literature. In finance, these techniques have been used in order to price and hedge exotic and American options via finite state Markov chain and binomial tree approximations; see e.g.\ \cite{GS:06,K:06,D:16}. As explained by\ \cite{kushner1984approximation} and~\cite{kushner2013numerical}, these approximations are linked to numerical analysis techniques such as the finite difference method. It is also relevant to mention quantization methods that address the optimal choice of the approximation grid on a finite time domain and in higher dimensional state spaces. Quantization has been employed to price American options by\ \cite{bally2005quantization}, and in the context of polynomial processes by\ \cite{Callegaroetal:17}. In all these cases, discretization happens at two levels: the discretization of the time domain, as it is performed in simulation algorithms, and the discretization of the space domain. We add to this literature by developing a cubature based discretization of stochastic models.

Cubature methods play a crucial role in numerous numerical algorithms. For instance, classical cubature techniques have been applied within the context of filtering in~\cite{arasaratnam2009cubature}. Additionally, the cubature formulas on Wiener space, developed by~\cite{lyonsvictoir}, have been used in multiple applications: in filtering problems by\ \cite{leelyons}, to calculate greeks of financial options by\ \cite{teichmann2006calculating}, and to numerically approximate solutions of Stochastic Differential Equations by\ \cite{bayer2008cubature} and~\cite{doersek2013cubature}, Backward Stochastic Differential Equations (BSDEs) by\ \cite{crisan2012solving,crisan2014second}, and Forward-Backward Stochastic Differential Equations (FBSDEs) by\ \cite{de2015cubature}. Cubature methods ease the calculation of conditional expectations, which are at the core of the above mentioned numerical problems. Contrary to the techniques mentioned in the previous paragraph where discretization is performed in the time and space domains, cubature on Wiener space discretizes path space directly. These cubature rules extend Tchakaloff's cubature theorem, as studied by\ \cite{putinar1997note} and\ \cite{bayer2006proof}, to the Wiener space of continuous paths. Our Markov cubature of polynomial processes provides a practically feasible variant of cubature of stochastic processes, as it is based on elementary matrix exponential calculus.

Our paper is organized as follows. In Section~\ref{sec:2} we define Markov cubature rules and provide some basic facts about polynomial processes. In particular, in Section~\ref{sec:MCandliftedMC} we explain why the notion of Markov cubature rule is too stringent in continuous time. In Section~\ref{sec:cont}, we give algebraic and geometric characterizations of continuous time Markov cubature rules for polynomial processes; see Theorem~\ref{thm:continuouscubature}. Motivated by this result we introduce, in Section~\ref{sec:approxMarkov}, a notion of {\em approximate} continuous time Markov cubature rule, and describe the quadratic programming problem through which it is obtained. The performance of these approximate Markov cubature rules is illustrated through numerical examples. Specifically, in Sections~\ref{sec:AmericanBS} and~\ref{sec:Jacobi} we use them to price American options in the Black--Scholes model and in a Jacobi model of exchange rates. In Section~\ref{sec:discrete}, we study existence of discrete time Markov cubature rules; see Theorem~\ref{thm:discretecubature}. In Section~\ref{sec:negativeweights} we discuss another possible relaxation of the Markov cubature problem by allowing negative weights. However, as we then illustrate, these negative weights are not suitable for numerical computations. The conclusions of our study are summarized in Section~\ref{sec:conclusions}. Appendix~\ref{sec:appendix1} presents results on asymptotic moments of polynomial processes needed throughout the paper, and Appendix~\ref{sec:appendix2} contains the proofs of all results in the main text.

We adopt the following notation: We write $\R_+$ for the set of nonnegative real numbers, $\R_{++}$ for the set of positive real numbers, and $\N$ for the set of positive natural numbers. For $N,M\in\N$, $\R^{N\times M}$ denotes the vector space of $N\times M$ matrices, and by convention $\R^N=\R^{N\times 1}$ consists of column vectors. Given $d\in\N$ and a set $E\subseteq \R^d$, we say that $q$ is a polynomial on $E$ if there exists a polynomial $p$ on $\R^d$ such that $q=p|_E$. Its degree is defined by ${\rm deg }\,q=\min\{{\rm deg }\, p:q=p|_E\}$. We let $\Pol(E)$ and $\Pol_n(E)$ denote the algebra of polynomials on $E$ and the vector space of polynomials on $E$ of degree less than or equal to $n$, respectively. For $N\in \N$ and a set $\mathcal A\subseteq\R^N$ we write $\conv(\mathcal A)$ for the convex hull of $\mathcal A$.


\section{Setup and overview}
\label{sec:2}
Fix a state space $E\subseteq\R^d$. We consider a c\`adl\`ag adapted process $X$ defined on a filtered measurable space $(\Omega,\Fcal,\Fcal_t)$, along with a family of probability measures $\P_x$, $x\in E$, such that $X$ is an $E$-valued Markov processes under each $\P_x$, starting at $X_0=x$. We assume that $X$ admits an extended generator $\Gcal$, whose domain contains all polynomials. That is, we assume
\[
p(X_t) - \int_0^t \Gcal p(X_s)\,ds \quad\text{is a $\P_x$-local martingale}
\]
for every $x\in E$ and every $p\in\Pol(\R^d)$. This implies in particular that $X$ is a semimartingale under each $\P_x$. Moreover, the positive maximum principle holds, in the sense that for any $p\in\Pol(\R^d)$,
\begin{equation} \label{PMP}
\text{if $p(x)=\max_E p$ for some $x\in E$, then $\Gcal p(x) \le 0$.}\footnote{Indeed, suppose $p(x)=\max_E p$, and assume for contradiction $\Gcal p(x)=\delta>0$. Define $M_t=p(X_t)-p(x)-\int_0^t \Gcal p(X_s)ds$ and $\tau=\inf\{t\colon \Gcal p(X_t)\le\delta/2\}$. Then, under $\P_x$, $M^\tau$ is a nonpositive local martingale with $M^\tau_0=0$, hence $M^\tau=0$. On the other hand, $M_{t\wedge\tau}\le-\int_0^{t\wedge\tau}\Gcal p(X_s)ds\le - (\delta/2)(t\wedge\tau)$, which is strictly negative for $t>0$. This contradiction proves $\Gcal p(x)\le0$.}
\end{equation}
In particular, $\Gcal p=0$ on $E$ whenever $p=0$ on $E$, which implies that $\Gcal$ is well-defined as an operator on $\Pol(E)$.\footnote{Indeed, if $p\in\Pol(\R^d)$ is a representative of $q=p|_E \in\Pol(E)$, we define $\Gcal q = \Gcal p|_E$, which is independent of the choice of representative~$p$.}


\subsection{Markov cubature rules}
\label{sec:MCandliftedMC}
Our goal is to construct a time-homogeneous Markov process with finite state space that approximates the process $X$. We base our approximation on moment conditions across initial states and times. With this goal is mind we make the following definition.\begin{definition} \label{D:nMC}
We say that a time-homogeneous Markov process $Y$ with finite state space $E^Y=\{x_1,\ldots,x_M\}\subseteq E$ defines an {\em $n$-Markov cubature rule for $X$ on $\T\subseteq[0,\infty)$} if
\begin{equation}\label{eq:Matchmoments}
\E_{x_i}[p(X_t)] = \E_{x_i}[p(Y_t)]
\end{equation}
holds for all $i=1,\ldots,M$, $t\in\T$, and $p\in\Pol_n(E)$.
\end{definition}

\begin{remark}
In condition~\eqref{eq:Matchmoments}, $\E_{x_i}[p(X_t)]$ denotes the expectation with respect to the probability measure $\P_{x_i}$ while $\E_{x_i}[p(Y_t)]$ denotes the expectation with respect to the probability measure $\P_{x_i}^Y$ associated to the finite state Markov process $Y$. We adopt this convention throughout the paper.\end{remark}

Suppose that $Y$ is a $n$-Markov cubature rule for $X$ on $\T$. The moment-matching condition~\eqref{eq:Matchmoments} can be rewritten as
\begin{equation}\label{E:momentmatch1}
\E_{x_i}[p(X_t)]=\sum_{j=1}^M p(x_j) \P_{x_i}^Y(Y_t=x_j)
\end{equation}
for all $i=1,\ldots,M$, $t\in\T$, and $p\in\Pol_n(E)$. Hence, for any $i=1,\ldots, M$ and  $t\in\T$, the points $x_1,\ldots,x_M$ together with the transition probabilities $ \P_{x_i}^Y(Y_t=x_1),\ldots, \P_{x_i}^Y(Y_t=x_M)$ define an $n$-cubature rule for the law of $X_t$ with respect to $\P_{x_i}$. We highlight that for Markov cubature rules, contrary to classical cubature rules, the matched moments depend on the cubature points, and the same points are used for all times $t\in\T$.  In addition, as stated in Theorem~\ref{thm:timeconsistency} below, the properties of the weights inherited by the Markov property of Y guarantee time-consistency features of these cubature rules for polynomial processes. This time-consistency is desirable to conduct path-dependent computations as the ones presented in the numerical examples in Section~\ref{sec:approxMarkov}.

We will also consider relaxed versions of $n$-Markov cubature rules. Indeed, it turns out that the notion of an $n$-Markov cubature rule is too stringent in general. To see why, suppose $X$ is given as the solution of an SDE of the form
\[
dX_t = \mu(X_t)\,dt + \sigma(X_t)\,dW_t.
\]
Under linear growth conditions on the coefficients, one has the estimate
\[
\E_x[\|X_t-x\|^4] \le \kappa (1+\|x\|^4)\, t^2, \quad 0\le t\le 1,
\]
for all $x\in E$, where $\kappa$ is a constant that only depends on $\mu$ and $\sigma$; see Problem~5.3.15 in~\cite{karatzas1991brownian}. If $Y$ is a $4$-Markov cubature rule for $X$ on $[0,\infty)$, this estimate carries over to $Y$, which in conjunction with the time-homogeneous Markov property yields
\[
\E_x[\|Y_t-Y_s\|^4] = \E_x\left[ \E_{Y_s}[\|Y_{t-s}-Y_0\|^4] \right] \le \kappa \Big(1+\max_{i=1,\ldots,M}\|x_i\|^4\Big)(t-s)^2
\]
for any $x\in E^Y$ and any $s\le t$ with $t-s\le1$. By Kolmogorov's continuity lemma, $Y$ then has a version with continuous paths, which forces it to be constant. Consequently, in the generic case, the diffusion $X$ will not admit any non-trivial $n$-Markov cubature rule on $[0,\infty)$, unless $n<4$. Moreover, by a similar argument, unless $X$ exhibits jumps, it is impossible to construct a non-trivial Markov process $Y$ with countable state space such that~\eqref{eq:Matchmoments}, with $n\geq 4$, holds for all initial conditions. This is a rather severe restriction.

One way to avoid this obstruction is to relax the exact moment matching condition~\eqref{eq:Matchmoments} and allow a process $Y$ whose moments approximate the moments of the original process $X$. This approach is explained in Section~\ref{sec:approxMarkov}. Another possibility is to replace $[0,\infty)$ with a discrete time set~$\T$, in which case one remains within the framework of Definition~\ref{D:nMC}. This approach is pursued in Section~\ref{sec:discrete}. A different relaxation is obtained if negative weights in the cubature rule are allowed. This approach is explained in Section~\ref{sec:negativeweights}.

We will study these relaxations of the Markov cubature problem for polynomial processes. This will allow us to employ algebraic considerations in our study. We give the basic properties of polynomial processes in the next subsection.


\subsection{Polynomial processes}

\begin{definition} \label{D:PP}
The operator $\Gcal$ is called {\em polynomial} if $\Gcal\Pol_n(E) \subseteq \Pol_n(E)$ for all $n\in\N$. In this case $X$ is called a {\em polynomial process}.
\end{definition}

\begin{remark}
In the present paper, $\Gcal$ is assumed to be the extended generator of some given Markov process $X$. We are not concerned with the question of existence of such a process given a candidate operator $\Gcal$. This issue is discussed in \citet{filipovic2016polynomial} for polynomial diffusions.
\end{remark}

If $X$ is a polynomial process, then all mixed moments of $X_t$ are polynomial functions of the initial state. More precisely, fix $n$ and denote by ${N_n}$ the dimension of $\Pol_n(E)$. Let $h_1,\ldots,h_{{N_n}}$ be a basis for $\Pol_n(E)$ and define
\begin{equation} \label{eq:Hn(x)}
\Hcal_n(x) = (h_1(x), \ldots, h_{{N_n}}(x))^\top.
\end{equation}
If $\Gcal$ is polynomial, one has
\begin{equation}\label{matrixG}
\Gcal \Hcal_n(x) = G_n^\top \Hcal_n(x)
\end{equation}
for some matrix $G_n\in\R^{N_n\times N_n}$, where $\Gcal$ acts componentwise on $\Hcal_n$. From this one obtains the following lemma.

\begin{lemma}\label{lem:momentsPP}
Assume $X$ is a polynomial process. Then for any polynomial $p\in\Pol_n(E)$ with coordinate representation $\vec p\in\R^{N_n}$, that is, $p(x)=\Hcal_n(x)^\top\vec p$, one has
\begin{equation}\label{matrixexptG}
\E_x[ p(X_t) ] = \Hcal_n(x)^\top e^{tG_n} \vec p.
\end{equation}
Thus the left-hand side is a polynomial in $\Pol_n(E)$ with coordinate representation $e^{tG_n}\vec p$.
\end{lemma}

\begin{remark}
As a consequence of Lemma~\ref{lem:momentsPP}, Markov cubature rules for polynomial processes are polynomial processes as well when $\T$ contains an interval around zero.
\end{remark}

We say that the time set $\T$ is stable under differences, if $t-s\in\T$ for all $s,t\in\T$ such that $s\leq t$. This property turns out to be useful for path-dependent computations involving polynomial processes, as we illustrate numerically in Section~\ref{sec:approxMarkov}. The reason is that stability under differences leads to the following time consistency result, which states that not only the one-dimensional marginals satisfy moment matching, but higher dimensional marginals do as well.

\begin{theorem}
\label{thm:timeconsistency}
Suppose that $X$ is a polynomial process and that $\T$ is stable under differences. Let $Y$ be a time-homogeneous Markov process with state space $E^Y=\{x_1,\ldots,x_M\}\subseteq E$. Then the process $Y$ is an $n$-Markov cubature rule for $X$ on $\T$ if and only if given $t_1,\ldots,t_l\in\T$ such that $0\leq t_1\leq\cdots\leq t_l$ and polynomials $p_1,\ldots,p_l\in\Pol_n(E)$ with $\prod_ip_i\in\Pol_n(E)$, we have	
\begin{equation}\label{eq:timeconsistency}
\E_x\left[\prod_{i=1}^lp_i(X_{t_i})\right]=\E_x\left[\prod_{i=1}^lp_i(Y_{t_i})\right]
\end{equation}
for all $x\in E^Y$.
\end{theorem}

\begin{remark}\label{rem:timesetsums}
Assume that $Y$ is an $n$-Markov cubature rule for a polynomial process $X$ on $\T$. Set $\overline{\T}=\{\sum_{i=1}^lt_i:t_i\in\T, l\in\N\}$. The time set $\overline{\T}$ is the smallest subset of $[0,\infty)$ that is stable under sums and contains $\T$.  The argument in the proof of Theorem~\ref{thm:timeconsistency} shows that $Y$ is also an $n$-Markov cubature rule for $X$ on $\overline{\T}$.
\end{remark}


\section{Continuous time Markov cubature for polynomial processes}
\label{sec:cont}

We assume hereafter that $X$ is a polynomial process, and fix $n\in\N$. We will study characterizations of continuous time $n$-Markov cubatures rules for $X$, namely $n$-Markov cubature rules on $[0,\infty)$. Even though, as explained in Section~\ref{sec:MCandliftedMC}, these cubature rules turn out to be too stringent in general, the results of this section motivate and facilitate the study of relaxed notions of Markov cubature in Sections~\ref{sec:approxMarkov},~\ref{sec:discrete} and~\ref{sec:negativeweights}.

We adopt the notation of Section~\ref{sec:2} but for simplicity we often omit the index $n$. Given points $x_1,\ldots,x_M\in E$ we denote by $H=H(x_1,\ldots,x_M)$ the $M\times N_n$-matrix whose elements are
\begin{equation}
\label{matrixH}
H_{ij}=h_j(x_i)
\end{equation}
for all $i=1,\ldots, M$ and $j=1,\ldots, N_n$. Notice that the  $i$-th row of  the matrix $H\in\R^{M\times N_n}$ is equal to $\Hcal_n(x_i)^\top$ as defined in~\eqref{eq:Hn(x)}.

By~\eqref{matrixG} and~\eqref{matrixexptG} we have
	\begin{align}
		\mathcal{G}h_j(x_i)&=(HG)_{ij}\label{matrixHG},\\	
		\E_{x_i}[h_j(X_t)]&=(H\exp(tG))_{ij}\label{matrixHexptG}	
	\end{align}
for all $i=1,\ldots, M$ and $j=1,\ldots, N_n$. Equations~\eqref{matrixHG}-\eqref{matrixHexptG} establish a relationship between the analytical calculation of the generator and semigroup acting on the function space of polynomials, and an algebraic calculation involving matrix multiplication.

Theorem~\ref{thm:continuouscubature} below is the main characterization theorem for the existence of a continuous time $n$-Markov rule. Before stating the theorem we recall that a \emph{transition rate matrix} is a matrix whose rows add up to zero and off-diagonal elements are nonnegative. We also need the following definition.

\begin{definition}\label{def:pointi_in}
We say that a vector $v\in\R^m$ points into $\conv(\{v_1,\ldots,v_n\})\subset\R^m$ at $v_i$ if there exist $(L_{i,j})_{j\neq i}\in \R_+^{m-1}$ such that
\begin{equation*}
 v=\sum_{j\neq i}L_{i,j}(v_j-v_i).
\end{equation*}
\end{definition}

\begin{theorem}
\label{thm:continuouscubature}
Given a set of points $E^Y=\{x_1,\ldots,x_M\}\subseteq E$ the following statements are equivalent.

\begin{enumerate}
\item\label{thm:continuouscubature:1} There exists a continuous time $n$-Markov cubature rule $Y$ with state space $E^Y$; see Definition~\ref{D:nMC}.
\item\label{thm:continuouscubature:2} Given $H$ as in~\eqref{matrixH}, $HG=LH$ for some transition rate matrix $L\in\R^{M\times M}$.
\item\label{thm:continuouscubature:3} Given $H$ as in~\eqref{matrixH}, $HG=LH$ for some matrix $L\in\R^{M\times M}$ with nonnegative off-diagonal elements.
\item\label{thm:continuouscubature:4} For each $x\in E^Y$ the vector $\Gcal \Hcal_n(x)$ points into $\conv(\{\Hcal_n(x_1),\ldots,\Hcal_n(x_M)\})$ at the point $\Hcal_n(x)$; see Definition~\ref{def:pointi_in}.
\end{enumerate}
	
If in addition $M=N_n$ and the matrix $H$ is invertible, there exists a Lagrange basis of $\Pol_n(E)$, $\widetilde{\beta}=(\widetilde{h}_1,\ldots,\widetilde{h}_{N_n})$, i.e.\ a basis with $\widetilde{h}_j(x_i)=\delta_{ij}$, and the above statements are equivalent to:
\begin{enumerate}
\setcounter{enumi}{4}
\item \label{thm:continuouscubature:5} $\Gcal \widetilde{h}_j(x_i)\geq 0$ for $i\neq j$.
\end{enumerate}
Moreover, when condition~\ref{thm:continuouscubature:2} is satisfied, $L$ can be taken as the transition rate matrix of the $n$-Markov cubature rule $Y$.
\end{theorem}

For the proof of Theorem~\ref{thm:continuouscubature} we will need the following lemma.

\begin{lemma}
\label{L:rowsLaddto0}
Suppose that $L$ is a matrix such that $HG=LH$. Then the rows of $L$ add up to zero.
\end{lemma}

As the proof shows, the conditions in Theorem~\ref{thm:continuouscubature} imply that if $Y$ is an $n$-Markov cubature rule then, for each $x\in E^Y$, the flow $(\E_{x}[\Hcal_n(X_t)])_{t\geq 0}$ never leaves the convex set $\conv(\{\Hcal_n(x_1),\ldots,\Hcal_n(x_M)\})$. Indeed, notice that $(\exp(tL))_{t\geq 0}$ is a transition semigroup and for all $i=1,\ldots, M$ we have
\[
	\E_{x_i}[\Hcal_n(X_t)]=\exp(tG^\top)\Hcal_n(x_i)=\text{$i$-th column of $H^\top\exp(tL^\top)$}.
\]
The points $\{\Hcal_n(x_1),\ldots,\Hcal_n(x_M)\}$ lie on the \emph{moment curve} $\Hcal_n(E)$ and correspond to the rows of $H$. Their convex hull represents all the possible initial distributions of a Markov chain with state space $\{\Hcal_n(x_1),\ldots,\Hcal_n(x_M)\}$.


\section{Approximate Markov cubature}
\label{sec:approxMarkov}

According to Theorem~\ref{thm:continuouscubature}, in order to find a continuous $n$-Markov cubature rule for a polynomial process $X$ one has to find points $x_1,\ldots,x_M\in E$ and a transition rate matrix $L$ such that
\[
HG=LH,
\]  
where $H$ is the matrix defined by \eqref{matrixH} and $G=G_n$ is the matrix of the generator of $X$ restricted to $\Pol_n(E)$ with respect to the basis $h_1,\ldots,h_{{N_n}}$; see~\eqref{matrixG}. As explained in Section~\ref{sec:MCandliftedMC}, it is actually impossible to solve this problem for polynomial diffusions if $n\ge 4$. In view of this restriction we instead consider the optimization problem
\begin{equation}\label{E:minHG}
\min\{\|HG-LH\|^2: \text{$L$ is a transition rate matrix}\},
\end{equation}
where the Frobenius norm is used, and where we have fixed the generator matrix $G$ and the points $x_1,\ldots,x_M$, hence the matrix $H$.\footnote{Recall that the Frobenius norm of a matrix $A$ is $\|A\|=\sqrt{\tr(AA^\top)}$.} The constraint that $L$ be a transition rate matrix can be written
\begin{align} 
&L_{ij}\ge 0, \quad i\neq j,\label{E:minHGR1onL}\\
&L{\rm 1}_M=0,\label{E:minHGR2onL}
\end{align}
where ${{\rm 1}_M}\in\R^M$ is a vector of ones. Through vectorization we can write this optimization problem as a quadratic programming problem. Indeed,  we have  
\[
\|HG-LH\|^2=\|HG\|^2+{\rm vec}(L)^\top(HH^\top\otimes \Id_M){\rm vec}(L)-2{\rm vec}(L)^\top{\rm vec}(HGH^\top),
\]
where ${\rm vec}(\,\cdot\,)$ is the vectorization operator, $\otimes$ the Kronecker product, and $\Id_M$ the $M$-dimensional identity matrix. In addition, the constrains~\eqref{E:minHGR1onL}-\eqref{E:minHGR2onL} on $L$ correspond to 
\begin{align} 
z^\top{\rm vec}({\rm e}_i{\rm e}_j^\top)&\geq 0, &i\neq j\label{E:minHGR1}\\
z^\top{\rm vec}({\rm e}_i{{\rm 1}_M}^\top)&=0, &i=1,\ldots, M\label{E:minHGR2}
\end{align}
where $z={\rm vec}(L)$ and the ${\rm e}_i$'s are the canonical basis vectors in $\R^M$. Therefore the minimization problem~\eqref{E:minHG} can be translated into the quadratic programming problem
\begin{equation}\label{eq:QLproblem}
\min\{z^\top(HH^\top\otimes \Id_M)z-2z^\top{\rm vec}(HGH^\top): \text{$z\in\R^{m\times m}$ satisfies \eqref{E:minHGR1}-\eqref{E:minHGR2}} \}.
\end{equation}
We will illustrate the performance of this type of finite state Markov approximation through numerical examples. 


\subsection{American option pricing in the Black--Scholes model}
\label{sec:AmericanBS}

We consider a Black--Scholes model where the financial asset's return process $X$ is a Brownian motion with drift. More precisely, $X$ is supposed to have risk-neutral dynamics of the form
\[
X_t=X_0+\left(r-\frac{\sigma^2}{2}\right)t+\sigma W_t
\]
where $r$ is the spot interest rate, $\sigma$ is the volatility of the returns and $W$ is a one-dimensional Brownian motion. In this model, the price at time $t=0$ of an American put option with maturity $T$, strike price $K$, and initial log-price $X_0+x$ is
\begin{equation}
\label{eq:PA}
P_{x}^{A}=\sup\{\E[{\rm e}^{-r\tau}\max(K-{\rm e}^{x+X_\tau},0)]:\text{$0\le \tau\le T$ a stopping time}\}.
\end{equation}
To approximate the value $P_{x}^{A}$ we proceed as follows. We fix equidistant points $x_1,\ldots,x_M$ on the truncated support of the process $X-X_0$ given by 
\[I=[(r-\sigma^2/2)T-3\sigma\sqrt{T},(r-\sigma^2/2)T+3\sigma\sqrt{T}].
\] 
We further fix $n\in\N$ and $h_1(x)=1,h_2(x)=x,\ldots,h_{n+1}(x)=x^n$ the standard monomial basis of $\Pol_n(\R)$. Let $L$ be the solution of the quadratic programming problem~\eqref{eq:QLproblem} and define $Y$ as the finite state process on $E^Y=\{x_1,\ldots,x_M\}$ with transition rate matrix $L$. For $N_{time}\in\N$ let $\Tcal=\{t_0=0,t_1\ldots,t_{N_{time}}=T\}$ be a uniform partition of the time horizon $[0,T]$. We define $\widetilde{P}^{A}=(\widetilde{P}_{x_1}^{A},\ldots,\widetilde{P}_{x_M}^{A})^\top$ by
\[
\widetilde{P}_{x_i}^{A}=\sup\{\E_{x_i}[{\rm e}^{-r\tau}\max(K-{\rm e}^{X_0+Y_\tau},0)]:\text{$0\le \tau\le T$ a stopping time with values in $\Tcal$}\}
\]
for $i=1,\ldots,M$. Since $Y$ is a finite state Markov process and we are only considering finitely many exercise times in $\Tcal$, the vector $\widetilde{P}^{A}$ can be computed through a very simple backward induction algorithm. This computation resembles the calculation of an American option price in a binomial tree approximation of the Black--Scholes model; see \citet{COX1979229}. Explicitly, we have $\widetilde{P}^{A}=V_0$ where 
\begin{equation}\label{eq:backrecursion}
\begin{split}
V_{t_{N_{time}}}&=\max(K-E,0)\\
V_{t_{i-1}} &=\max(\max(K-E,0),\exp(-r\Delta)\exp(\Delta L)V_{t_i}),\quad i=1,\ldots,N_{time}
\end{split}
\end{equation}
with  $E=\exp(X_0)(\exp(x_1),\ldots,\exp(x_M))^\top\in\R^M$ and $\Delta=T/N_{time}$. Observe that this computation gives simultaneously all the values $\widetilde{P}_{x_i}^{A}$ for $i=1,\ldots,M$. For any initial value $x\in I$ we can simply perform an interpolation in order to approximate $P_{x}^{A}$. Moreover, we have that $P_{x}^{A}/{\rm e}^{x}$ is the price of an American put option with strike $K{\rm e}^{-x}$ and initial underlying log-price $X_0$. Hence, the same approximate cubature rule can be used to price American options for several initial values of the log-price and for different strikes. This observation remains valid in any stochastic volatility model framework as long as the dynamics of the volatility process are independent of the initial value of the spot price.

To illustrate the performance of our method we consider the parameters $r=0.06$, $\sigma=0.4$, $X_0=\log(K)=\log(100)$ and $T=0.5$. We compute the approximate American put option prices with $M=40$ cubature points, $n=4$ moments and $N_{time}=1000$ time steps. We compare these prices with the benchmark prices obtained with a 1000--time step binomial tree approximation of the Black--Scholes model. The results are reported in Table~\ref{T:PricesBS}. We find that with these parameters our approximate Markov cubature method has a mean relative difference with respect to the benchmark binomial tree prices of the order $10^{-4}$. The choice of $M$ and $n$ is made to achieve this level of accuracy in a comparable amount of time with as few cubature points and moments as possible. The colormap in Figure~\ref{fig:colormapBS} shows the off-diagonal values of the transition rate matrix $L$. We observe in particular that the majority of nonzero transition rates in the approximate Markov cubature rule occur around the diagonal, hence the process $Y$ has a multinomial tree structure. Also the transition rates decrease as we approach the limits of the interval $I$. The high transition rates close to the limit points of the interval are a boundary effect as a consequence of the truncation of the domain of $X-X_0$. The running time to find the transition rate matrix $L$ by solving the optimization problem \eqref{E:minHG} in Matlab on a 2.3 GHz Intel Core i5 CPU, is approximately 0.75 seconds. Once the transition rate matrix is obtained, the computation of the American option prices using the recursive algorithm \eqref{eq:backrecursion}, for a given maturity, a given strike, and all initial prices, is almost instantaneous and takes only about 0.004 seconds. To illustrate the influence of the moments, we plot in Figure~\ref{fig:American_BS} the American put option prices $\widetilde{P}^A$ for $M=40$ and different values of $n$. These prices are compared with the benchmark values obtained with a 1000 time step binomial approximation of the Black--Scholes model on the 40 points of the log-price grid.


\subsection{American option pricing in a Jacobi exchange rate model}
\label{sec:Jacobi}
Suppose that $S_t=\exp(X_t)$ represents the exchange rate between two currencies at time $t$. Inspired by \cite{DeJongetal01} and \cite{LarsenSorensen07}, we model $X$ with a Jacobi diffusion of the form
\begin{equation}\label{eq:FXrate}
dX_t=\kappa(\theta-X_t)dt+\sigma\sqrt{(X_t-x_{min})(x_{max}-X_t)}dW_t,
\end{equation}
for given parameters $-\infty< x_{min}<x_{max}<\infty$, $\theta\in [x_{min},x_{max}]$, $\kappa, \sigma>0$. We assume that the domestic interest rate $r$ is constant and the foreign interest rate $r^f_t$ is such that~\eqref{eq:FXrate} describes the risk neutral dynamics of the log exchange rate (for details see \cite{DeJongetal01,LarsenSorensen07}). In this model the exchange rate between the currencies stays bounded between $\exp(x_{min})$ and $\exp(x_{max})$. As in the previous example, we consider an American exchange option with payoff of the form $\max(K-S,0)$ and maturity $T$. We fix equidistant points $x_1,\ldots,x_M$ on the support of $X$ given by $[x_{min},x_{max}]$ and proceed precisely as in the previous example. Namely, we approximate $P^A_x$, as in   \eqref{eq:PA}, using the vector $\widetilde P^A$. To compute $\widetilde P^A$ we use the recursive algorithm in \eqref{eq:backrecursion}.

For our numerical illustration we consider the following parameters: $r=0$, $\kappa=1$, $\theta=0.5$, $x_{min}=0$, $x_{max}=1$, $X_0=0$, $K=\exp(0.5)$ and $T=0.5$. We compute the approximate American put option prices with $M=40$ cubature points, $n=4$ moments and $N_{time}=1000$ time steps. We compare these prices with the benchmark prices obtained with a 1000--time step Longstaff--Schwartz algorithm; see~\cite{Longstaff01valuingamerican}. The results are reported in Table~\ref{T:PricesJacobi}. Our experiments suggest that the approximate Markov cubature method can lead to significant speed-up compared to the simulation based Longstaff--Schwartz approach. The colormap in Figure~\ref{fig:colormapJacobi} shows the off-diagonal values of the transition rate matrix $L$. In particular we verify a multinomial nature of our approximate Markov cubature. The running times to find the transition rate matrix $L$ and to compute American option prices are comparable to those reported in Section~\ref{sec:AmericanBS}. To illustrate the influence of the moments, we plot in Figure~\ref{fig:American_Jacobi} the American put option prices $\widetilde{P}^A$ for $M=40$ and different values of $n$ and compare them with the benchmark values obtained with the Longstaff--Schwartz method for the values of $x$ in Table~\ref{T:PricesJacobi}. 


\section{Discrete time Markov cubature}
\label{sec:discrete}
The construction of a discrete time $n$-Markov cubature rule for $X$ (see Theorem~\ref{thm:discretecubature} below) will use cubature methods over the asymptotic moments.
According to Theorem~\ref{thm:longrunmoments}, all the asymptotic moments of order less than or equal to $n$ exist if and only if the following condition holds.
\begin{enumerate}[label=(H\arabic{*}), ref=(H\arabic{*})]
\item \label{assumptionG} For all nonzero eigenvalues $\lambda$ of $G$, we have that $Re(\lambda)< 0$ and the eigenvalue 0 is a semisimple eigenvalue, i.e.\ its algebraic and geometric multiplicities coincide.
\end{enumerate}
In this case, we denote these \emph{asymptotic moments} by
\begin{equation}
\label{longrunmoments}
\mu_j(x)= \lim_{t\to\infty}\E_x[h_j(X_t)].
\end{equation}

To use classical cubature rules, we would like the asymptotic moments~\eqref{longrunmoments} to be independent of $x$. According to Corollary~\ref{cor:independencemomentsonx}, this is the case under the following assumption, which is a stronger condition than~\ref{assumptionG}.
\begin{enumerate}[label=(H\arabic{*}), ref=(H\arabic{*})]
\setcounter{enumi}{1}
\item \label{assumptionG_2} For all nonzero eigenvalues $\lambda$ of $G$, we have that $Re(\lambda)< 0$ and the eigenvalue 0 is a simple eigenvalue, i.e.\ its algebraic  (and hence geometric) multiplicity is 1.
\end{enumerate}
In this case we write the asymptotic moments~\eqref{longrunmoments} simply as $\mu_1,\ldots,\mu_{N_n}$. In conjunction with~\ref{assumptionG_2}, we will make the following assumption throughout this section.

\begin{enumerate}[label=(H\arabic{*}), ref=(H\arabic{*})]
\setcounter{enumi}{2}
\item \label{assumption_longrunmoments} There exist points $x_1,\ldots,x_M\in E$ and $w\in\R_{++}^M$ such that
\begin{equation}
\label{eq:longrunweights}
\mu_j=\sum_{i=1}^M w_{i}h_j(x_i)
\end{equation}
for all $j=1,\ldots, N_n$.
\end{enumerate}

\begin{remark}
As a consequence of condition~\ref{assumption_longrunmoments} the weights add up to one. Indeed, suppose that the constant polynomial can be written as $\mathbf{1}=\sum_{j=1}^{N_n}v_j h_j(x)$. Then by \eqref{longrunmoments} and \eqref{eq:longrunweights} we deduce
\[
1=\sum_{j=1}^{N_n}v_j \mu_j = \sum_{i=1}^M w_{i} \left(\sum_{j=1}^{N_n}v_j  h_j(x_i)\right) =  \sum_{i=1}^M w_{i}.
\]
Hence, condition~\ref{assumption_longrunmoments} states that the asymptotic moments~\eqref{longrunmoments} belong to $\conv(\Hcal_n(E))$. As $\E_x[h_j(X_t)]$ belongs to $\conv(\Hcal_n(E))$ for all $x\in E$, $t\geq 0$ and $j=1,\ldots, N_n$ (see~\cite{putinar1997note},~\cite{bayer2006proof}), this would be the case if for instance $\conv(\Hcal_n(E))$ is closed. It would also hold if the asymptotic moments are the moments of a probability distribution; see Proposition~\ref{prop:momentlimitdist}. Additionally, as the weights $w$ in~\ref{assumption_longrunmoments} are strictly positive, there does not exist a strict subset $C\subsetneq \{x_1,\ldots,x_M\}$ such that $\conv(\Hcal_n(C))$ contains all the asymtotic moments~\eqref{longrunmoments}.
\end{remark}

Theorem~\ref{thm:discretecubature} below is the main theorem of this section.

\begin{theorem}
\label{thm:discretecubature}
Assume that~\ref{assumptionG_2} and~\ref{assumption_longrunmoments} hold. Suppose additionally that for the points $x_1,\ldots,x_M$ in~\ref{assumption_longrunmoments}, the matrix $H$ given by~\eqref{matrixH} satisfies $\rk(H)=N_n$ . Then, for $\Delta$ large enough, there exists a $n$-Markov cubature rule for $X$ on $\T=\{l\Delta:l\in\N\}$ with state space $E^Y=\{x_1,\ldots,x_M\}$.
\end{theorem}

To prove Theorem~\ref{thm:discretecubature} we need the following lemma.

\begin{lemma}
\label{lem:existenceofQ}
Suppose the hypotheses of Theorem~\ref{thm:discretecubature} hold. Then, for $t$ sufficiently large, there exists a probability matrix $Q(t)$ with positive entries such that $H\exp(tG)=Q(t)H$.
\end{lemma}

The proofs of Theorem~\ref{thm:discretecubature} and Lemma~\ref{lem:existenceofQ} suggest a possible procedure for finding discrete cubature rules. In practice, denoting $H^{+}$ the pseudo-inverse matrix of $H$, one searches for large times $t$ so that the matrix $Q(t)=H\exp(tG)H^+$ has positive entries. One then takes $\Delta$ to be the first time such that the entries of $Q(\Delta)$ are nonnegative, in which case $Q(\Delta)$ is the transition probability matrix of a discrete cubature rule with time lag $\Delta$.

The following remark shows that the existence of discrete time Markov cubature rules is true under more general hypotheses.

\begin{remark}
Assume that~\ref{assumptionG} holds. Suppose additionally that there exist points
$$x_1,\ldots,x_M\in E$$
and
$$W=(w_{ij})_{i,j=1}^M\in\R_{++}^{M\times M}$$
such that 	
\begin{equation*}
\mu_j(x_k)=\sum_{i=1}^M w_{ki}h_j(x_i)
\end{equation*}
for all $j=1,\ldots, N_n$ and $k=1,\ldots M$, with $\mu_j$ as in~\eqref{longrunmoments}.  The proof of Theorem~\ref{thm:discretecubature} shows that, if the matrix $H=H(x_1,\ldots,x_M)$, defined in~\eqref{matrixH}, satisfies $\rk(H)=N_n$, then the conclusion of Theorem~\ref{thm:discretecubature} holds.
\end{remark}

\begin{example}
Discrete cubature rules can be useful to perform computations in longer time periods, for example the calculation of prices of European options. To illustrate this we consider the exchange rate model and parameters described in Section~\ref{sec:Jacobi}. To approximate the price of a European put option $P^{E}_x=\E_x[{\rm e}^{-rT}\max(K-{\rm e}^{x+X_T},0)]$ with maturity $T=1$, strike $K=\exp(0.5)$, and initial log rate equal to $x$, we proceed as follows. We fix a regular partition $x_1,\ldots,x_M$ of the support of $X$. In this case the conditions of Theorem~\ref{thm:discretecubature} are satisfied. The asymptotic moments are
\[
\mu_1=\E[1]=1,\quad \mu_2 = \theta= 0.5, \quad \mu_3= 0.3333, \quad \mu_4 = 0.25, \quad \mu_5 = 0.2.
\]
We observe that for $T=1$ there is transition rate matrix $P$ such that $H\exp(TG)=PH$. We approximate the price of the European put options for the points on the partition using the vector 
\[
\widetilde{P}^E=\exp(-rT)PV 
\]
where $V=max(K-F,0)$ and $F=\exp(X_0)(\exp(x_1),\ldots,\exp(x_M))^\top\in\R^M$. Figure~\ref{fig:European_Jacobi} shows these approximate prices along with the prices obtained using Monte-Carlo simulation.
\end{example}


\section{Markov cubature with negative weights}
\label{sec:negativeweights}
In this section we explore yet another possible relaxation of the Markov cubature problem. We first recall the definition of the mapping $\Hcal_n$ in~\eqref{eq:Hn(x)}. In the spirit of~\cite{bayer2006proof}, observe that $n$-Markov cubatures rules for the process $X$ correspond naturally to 1-Markov cubature rules for the process $\overline{X}=\Hcal_n(X)$, and that the state space of $\overline{X}$ is $\Hcal_n(E)$, which lies on the moment curve $\Hcal_n(\R^d)$. It can be shown that $\overline{X}$ is a polynomial process on $\Hcal_n(E)$; see \citet[Theorem 4.2]{FilipovicLarson:17}. Hence, the study of $n$-Markov cubatures rules for polynomial processes can be reduced to the study of 1-Markov cubature rules by increasing the complexity of the state space. 

These observations suggest the following alternative way to relax the notion of Markov cubature in continuous time. For each $x\in E$, consider the process $Z^{\overline{x}}_t=\E_{\overline{x}}[\overline{X}_t]$. Due to Lemma~\ref{lem:momentsPP}, the process $Z^{\overline{x}}_t$ solves the ODE
\begin{equation}\label{eq:ODE lifted}
dZ^{\overline{x}}_t = G^\top Z^{\overline{x}}_t \,dt, \quad Z_0 = \overline x.
\end{equation}
While $\overline X$ is only well-defined for initial conditions $\overline x\in\Hcal_n(E)$, whose geometry is highly complex in general, the solution $Z^{\overline x}$ of~\eqref{eq:ODE lifted} admits any point $\overline x\in\R^{N_n}$ as initial condition. Therefore, one could seek continuous 1-cubature rules for the deterministic process $Z$ on $\R^{N_n}$ instead of $\overline{X}$ on $\Hcal_n(E)$. In view of Theorem~\ref{thm:continuouscubature} this amounts to finding points $z_1,\ldots,z_R\in\R^{N_n}$ and a transition rate matrix $L\in\R^{R\times R}$ such that
\begin{equation}\label{eq:SG=LS}
SG=LS,
\end{equation}
where $S\in\R^{R\times N_n}$ is a matrix whose rows are $z_1^\top,\ldots,z_R^\top$. Suppose we can find such a matrix of the form $S=\widetilde{S}H$, with a matrix $\widetilde{S}\in \R^{R\times M}$ of rank $M$. Then there also exists $A\in\R^{M\times R}$ such that $A\widetilde S=\Id_M$ the $M$-dimensional identity matrix, and we can rewrite~\eqref{eq:SG=LS} as
\begin{equation}\label{eq:HG=tildeLH}
HG=\widetilde L H,
\end{equation}
where $\widetilde L= AL\widetilde S\in \R^{M\times M}$. The matrix $\widetilde L$ is not necessarily a transition rate matrix. Nevertheless, due to Lemma~\ref{lem:momentsPP} and~\eqref{eq:HG=tildeLH} we have
\begin{equation}\label{eq:negprobas}
\E_{x_i}[p(X_t)]=\sum_{j=1}^M ({\rm e}^{t\widetilde L})_{i,j}p(x_j),\quad \text{for $i=1,\ldots,M$ and $p\in\Pol_n(E).$}
\end{equation}
Hence, the {\it pseudo transition rate matrix} $\widetilde L$, defines a {\it pseudo Markov cubature rule} with weights that might be negative. These possibly negative weights can be interpreted as the negative ``probabilities'' appearing in the {\it pseudo transition probability matrix} ${\rm e}^{t\widetilde L}$. Therefore, the limitation posed by Kolmogorov's continuity lemma disappears in a framework with negative ``probabilities''. As we will illustrate below, however, negative weights are not compatible with fundamental results such as the dynamic programming principle underlying the pricing of American options. For this reason, we do not pursue this relaxation of the Markov cubature problem.

\begin{example}
To illustrate why negative weights are not compatible with the dynamic programming principle, we consider the setup of Section \ref{sec:AmericanBS} and the same Black--Scholes model parameters. In particular, we employ $M=40$ cubature points, $n=4$ moments and $N_{time}=1000$ to approximate the American put option prices. We compare in Figure \ref{fig:American_negative} the results obtained by using the solution $L$ of the quadratic programming method described in Section~\ref{sec:approxMarkov} with those obtained with a matrix $\widetilde L$ that solves equation~\eqref{eq:HG=tildeLH}. This figure clearly demonstrates that the relaxation with negative ``probabilities'' is not useful for probabilistic applications.
\end{example}

\section{Conclusions}
\label{sec:conclusions}
In this paper we study discretizations of polynomial processes, via moment conditions, using finite state Markov processes. We call these discretizations Markov cubature rules. The polynomial property allows us to conduct our analysis using algebraic techniques; see Theorems~\ref{thm:continuouscubature} and~\ref{thm:discretecubature}. Due to Kolmogorov's continuity lemma the moment matching conditions in continuous time for polynomial diffusions are too stringent. We study instead relaxed versions of Markov cubature rules. A possible relaxation allowing negative transition ``probabilities'' shows not to be useful for probabilistic applications; see Section~\ref{sec:negativeweights}. We instead retain two other relaxations of the Markov cubature problem that are more useful. In Section~\ref{sec:approxMarkov} we show how to find approximate Markov cubature rules by means of a quadratic programming problem. We then illustrate with examples how to employ this approximation to solve time-dependent problems, like the valuation of American options. In these examples the method performs well. Then, in Section~\ref{sec:discrete} we discuss conditions on the asymptotic moments that allow the construction of discrete time Markov cubature rules; see Theorem~\ref{thm:discretecubature}. We also illustrate through a numerical example the use of these discrete rules on longer time grids. A systematic analysis of computational cost, accuracy, and convergence falls outside the scope of the present paper, but is an interesting topic for future research.

\appendix


\section{Asymptotic moments of polynomial processes}
\label{sec:appendix1}
Suppose that $X$ is a polynomial process with extended generator $\Gcal$ and state space $E$. Fix $n\in\N$ and let $G$ be the matrix of $\Gcal$ restricted to $\Pol_n(E)$ with respect to a basis $\beta=(h_1,\ldots,h_{N_n})$ of $\Pol_n(E)$.

The following theorem shows that Hypothesis~\ref{assumptionG} is equivalent to the existence of asymptotic moments of order $n$.

\begin{theorem}
\label{thm:longrunmoments}
The following are equivalent:

\begin{enumerate}
\item\label{thm:longrunmoments:1} Hypothesis~\ref{assumptionG} holds.
\item\label{thm:longrunmoments:2} The sequence of matrices $(\exp(tG))_{t\geq 0}$ converges as $t\to\infty$.
\item\label{thm:longrunmoments:3} $\E_x[h_j(X_t)]$ converges as $t\to\infty$ for all $x\in E$  and $j=1,\ldots, N_n$.
\item\label{thm:longrunmoments:4} $\E_x[p(X_t)]$ converges as $t\to\infty$ for all $x\in E$ and $p\in\Pol_n(E)$.
\end{enumerate}

\end{theorem}

\begin{proof}
\ref{thm:longrunmoments:1}$\Leftrightarrow$\ref{thm:longrunmoments:2} Suppose that $G=VJV^{-1},$ where $J$ is the (complex) Jordan normal form of $G$. We have that $(\exp(tG))_{t\geq 0}$ converges as $t\to\infty$ if and only if $(\exp(tJ))_{t\geq 0}$ converges as $t\to\infty$. Additionally, $(\exp(tJ))_{t\geq 0}$ converges as $t\to\infty$ if and only if $\exp(tJ_i)$ converges for all $i$, where the $J_i$'s are the Jordan blocks of the matrix $J$.

Each $J_i$ is of the form $J_i=\lambda_i \Id+N_i$ where $\lambda_i$ is an eigenvalue of $G$, $\Id$ is the identity matrix and $N_i$ is a nilpotent matrix. Therefore, $\exp(tJ_i)=\exp(t\lambda_i)p_i(tN_i),$ with $p_i$ a polynomial. We remark that $p_i\equiv 1$ if and only if $N_i=0$, and $p_i(tN_i)$ is not a constant polynomial in $t$ if $N_i\neq 0$.

Hypothesis~\ref{assumptionG} holds if and only if $Re(\lambda_i)<0$ for all $i$ such that $\lambda_i\neq 0$ and if $\lambda_i=0$, $N_i=0$. These observations imply the equivalence between \ref{thm:longrunmoments:1} and \ref{thm:longrunmoments:2}.

\ref{thm:longrunmoments:2}$\Rightarrow$\ref{thm:longrunmoments:3}  Suppose that the matrices $(\exp(tG))_{t\geq 0}$ converge to a matrix $\widetilde{P}\in\R^{N_n\times N_n}$ as $t\to\infty$. By~\eqref{matrixexptG}, we have that
\begin{displaymath}
\lim_{t\to\infty}\E_x[h_j(X_t)]=\sum_{i=1}^{N_n}\widetilde{P}_{ij}h_i(x)
\end{displaymath}
for all $j=1,\ldots, N_n$ and $x\in E$. Hence \ref{thm:longrunmoments:3} holds.

\ref{thm:longrunmoments:3}$\Leftrightarrow$\ref{thm:longrunmoments:4}  This follows from the fact that $h_1,\ldots,h_{N_n}$ is a basis of $\Pol_n(E)$.

\ref{thm:longrunmoments:3}$\Rightarrow$\ref{thm:longrunmoments:2} Suppose now that $\E_x[h_j(X_t)]$ converges for all $x\in E$ and $j=1,\ldots, N_n$, as $t$ goes to infinity.

We claim that there exists $N_n$ points, $x_1,\ldots,x_{N_n}\in E$, such that for all $p\in\Pol_n(E)$
\begin{equation}
\label{eq:nicepoints}
p(x_i)=0\text{ for all $i$ }\Rightarrow p\equiv 0.
\end{equation}

Assume for the sake of contradiction that there are no points $x_1,\ldots,x_{N_n}\in E$ such that~\eqref{eq:nicepoints} holds. Let $p_1(x)\neq 0$ be a polynomial on $E$ and $x_1\in E$ such that $p_1(x_1)\neq 0.$  By assumption, we can find $p_2\in\Pol_n(E)$ and $x_2\in E$ such that $p_2(x_1)=0$ and $p_2(x_2)\neq 0.$  Recursively, we would be able to construct points $x_1,\ldots,x_{N_n}$, and polynomials $p_1,\ldots,p_{N_n}$ such that
\begin{equation}
\label{eq:nicepolynomials}
p_i(x_i)\neq0\text{ and }p_i(x_j)=0\text{ for }j<i.
\end{equation}
These polynomials would be linearly independent and hence a basis of $\Pol_n(E)$.

Assume that $p\in\Pol_n(E)$ satisfies $p(x_i)=0\text{ for all }i.$ As $p$ is a linear combination of the polynomials $p_i$ we would conclude by~\eqref{eq:nicepolynomials} that all the coefficients of the linear combination are equal to zero and $p$ is zero everywhere, a contradiction.

Hence we can always find $x_1,\ldots,x_{N_n}\in E$ such that~\eqref{eq:nicepoints} holds. These points allow us to define a norm on the space $\Pol_n(E)$ by
	$$\|p\|_1=\sup_{i}|p(x_i)|.$$
Another norm is given by
	$$\|p\|_2=\sup_i|\lambda_i|$$
where $p=\sum_j\lambda_jh_j$. As these norms are equivalent, convergence of a sequence of polynomials on $x_1,\ldots,x_{N_n}$ implies convergence of the coefficients. The coefficients of the polynomials of the form $\E_x[h_j(X_t)]$ are entries of the matrix $\exp(tG)$. Hence \ref{thm:longrunmoments:2} holds.
\end{proof}

In general, these asymptotic moments might depend on $x$. In fact we have the following proposition.

\begin{proposition}
Suppose that Hypothesis~\ref{assumptionG} holds. Let $G=V J V^{-1}$ be the canonical Jordan decomposition of $G$, with $V$ the matrix of generalized eigenvectors. Then
\begin{equation}
\label{eq:limitexptG}
\lim_{t\to\infty}e^{tG}=\sum_{i=1}^l v_ir_i,
\end{equation}
where the vectors $v_1,\ldots,v_l$ are the eigenvectors of $G$ corresponding to the eigenvalue 0 and $r_1,\ldots, r_l$ are the rows of $V^{-1}$. Moreover, the asymptotic moments~\eqref{longrunmoments} are given by
\begin{equation}
\label{eq:dependencexlimitingmoments}
(\mu_1(x),\ldots,\mu_{N_n}(x))=\sum_{i=1}^l \Hcal_n(x)^\top v_ir_i.
\end{equation}
\end{proposition}

\begin{proof}
The proof of the equivalence between \ref{thm:longrunmoments:1} and \ref{thm:longrunmoments:2} in Theorem~\ref{thm:longrunmoments} shows that Hypothesis~\ref{assumptionG} implies that
\begin{displaymath}
\lim_{t\to\infty}e^{tG}=V\begin{pmatrix}\Id&0\\0&0\end{pmatrix}V^{-1},
\end{displaymath}
where the identity matrix $\Id$ comes from the block corresponding to the eigenvalue $0$. This implies~\eqref{eq:limitexptG}. Moreover,~\eqref{eq:limitexptG},~\eqref{matrixHexptG} and~\eqref{longrunmoments} imply~\eqref{eq:dependencexlimitingmoments}.
\end{proof}

An immediate corollary of these results characterizes the case when the asymptotic moments are independent of $x$.

\begin{corollary}
\label{cor:independencemomentsonx}
Hypothesis~\ref{assumptionG_2} holds if and only if the asymptotic moments $\mu_1(x),\ldots,\mu_{N_n}(x)$ as defined in~\eqref{longrunmoments} exist and they are independent of $x$, i.e. constant on $E$.
\end{corollary}

\begin{proof}
We already have the equivalence between Hypothesis~\ref{assumptionG} and the existence of the asymptotic moments by Theorem~\ref{thm:longrunmoments}. Moreover, observe that~\eqref{eq:dependencexlimitingmoments} in the previous proposition implies that for all $j=1,\ldots, N_n$
\begin{displaymath}
\mu_j(x)=\sum_{i=1}^l r_i(j)\widetilde{h}_i(x),
\end{displaymath}
where the eigen-polynomials $\widetilde{h}_1,\ldots,\widetilde{h}_l$ (corresponding to the eigenvalue 0 of $\Gcal$) are given by
$$\widetilde{h}_i(x)=\Hcal_n(x)^\top v_i,$$
for all $i=1,\ldots,l$. These polynomials are linearly independent, as polynomials in $\Pol_n(E)$. This linear independence implies that $\mu_j(x)$ is constant on $E$ for all $j$ if and only if $l=1$.
\end{proof}

\begin{example}
Suppose that $X$ is a polynomial \textit{martingale}. This holds when $G_1=0$, where $G_1$ is the matrix of the generator restricted to the space $\Pol_1(E)$. A particular example is geometric Brownian motion. In this case we have that $\E_x[X_t]=x$ for all $t\geq 0$ and $x\in E$, and hence,
 \begin{displaymath}
 \lim_{t\to\infty} \E_x[X_t]=x.
\end{displaymath}
In this example, $0$, as an eigenvalue of $G_1$, has algebraic multiplicity 2.
\end{example}

\begin{example}
Suppose that $d=1$ and $\Gcal f(x)=-xf'(x)+x^2f''(x)$. Then
 \begin{displaymath}
 \lim_{t\to\infty} \E_x[X_t]=0;\quad\lim_{t\to\infty} \E_x[X^2_t]=x^2.
\end{displaymath}
In this example, $0$ has multiplicity 1 as an eigenvalue of $G_1$ (the matrix of the generator $\Gcal$ restricted to $\Pol_1(E)$.) However, 0 has algebraic multiplicity 2 as an eigenvalue of $G_2$ (the matrix of the generator $\Gcal$ restricted to $\Pol_2(E)$). The operator $\Gcal$ is the infinitesimal generator of the process $X_t=X_0{\rm e}^{-2t+\sqrt{2}B_t}$, where $B$ is a Brownian motion. In this case, $X_t$ is a supermartingale converging to zero in expectation while $X_t^2=X_0^2{\rm e}^{-4t+2\sqrt{2}B_t}$ is a martingale starting at $X_0^2$.
\end{example}

The following proposition gives sufficient conditions under which the limiting moments $\mu_j(x)$ are moments of a positive Borel measure.

\begin{proposition}
\label{prop:momentlimitdist}
Let $G_{n+1}$ be the matrix of the generator restricted to the space $\Pol_{n+1}(E)$ with respect to an extended basis $\widetilde{\beta}=(h_1,\ldots,h_{N_n},\ldots,h_R)$ of $\Pol_{n+1}(E)$. Assume that~\ref{assumptionG} holds for $G_{n+1}$. Then for all $x\in E$ there exists a positive Borel measure $\pi_x$ such that
\begin{equation}
\label{eq:10}
\int_Eh_j(y)\pi_x(dy)=\mu_j(x)
\end{equation}
for all $j=1,\ldots, N_n$.
\end{proposition}

\begin{proof}
Let $x\in E$ and $j=1,\ldots, N_n$ be fixed. We have by Theorem~\ref{thm:longrunmoments} that $\E_x[f(X_t)]$ converges as $t\to\infty$ for any polynomial $f\in\Pol_{n+1}(E)$. Define $Y_t= h_j(X_t).$  De La Vall\'{e}e-Poussin's theorem implies that $(Y_t)_{t\geq 0}$ is uniformly integrable. Additionally, we have that the sequence of Borel probability measures on $E$ given by $(\P_x\circ X_t^{-1})_{t\geq 0}$ is tight.

Let $\pi_x$ be an accumulation Borel probability measure of this sequence. We conclude that~\eqref{eq:10} holds. Indeed, assume with out loss of generality that $\P_x\circ X_t^{-1}$ converges in distribution to  $\pi_x$. By Fatou's lemma
\begin{align*}
\int_E|h_j(y)|\pi(dy)&=\int_0^\infty \pi(|h_j(y)|>z)dz\\
&\leq \liminf_{t\to\infty}\int_0^\infty \P_x(|Y_t|>z)dz\\
&=\liminf_{t\to\infty}\E[|Y_t|]<\infty.
\end{align*}
Therefore, given $j=1,\ldots N_n$ and $\epsilon>0$, there exist constants $C,T>0$ such that $\E_x[|Y_t|1_{|Y_t|>C}]<\epsilon$ for all $t\geq 0$, $\int_{|h_j(y)|>C}|h_j(y)|\pi(dy)<\epsilon$ and for $t\geq T$
\begin{displaymath}
\left|\E_x[Y_t1_{|Y_t|\leq C}]-\int_{|h_j(y)|\leq C}h_j(y)\pi(dy)\right|<\epsilon.
\end{displaymath}
Hence, for $t\geq T$
\begin{align*}
\left|\E_x[Y_t]-\int_EH_j(y)\pi_x(dy)\right|&\leq \left|\E_x[Y_t1_{|Y_t|\leq C}]-\int_{|h_j(y)|\leq C}h_j(y)\pi(dy)\right|\\
&+\E_x[|Y_t|1_{|Y_t|>C}]+\int_E|h_j(y)|1_{|h_t(y)|>C}\pi(dy)\\
&\leq 3\epsilon.
\end{align*}
Since $\epsilon>0$ was arbitrary we obtain~\eqref{eq:10}; see also Theorem 3.5 in~\cite{billingsley}.
\end{proof}

\begin{remark}
In the proof of Proposition~\ref{prop:momentlimitdist}, the existence of the asymptotic moments up to order $n+1$ is simply used in order to apply De La Vall\'{e}e-Poussin's theorem and to deduce uniform integrability.
\end{remark}

In some cases the measure $\pi_x$ of the proposition above is not necessarily an invariant measure. The following example illustrates this.
\begin{example}
Suppose that $X$ is an exponential Brownian motion. In particular $X$ is a martingale and $\E_x[X_t]=x$ for all $t,x\geq 0$. Hence, $\mu_2(x)=x$, where $\mu_2$ is the asymptotic mean. In this case $\pi_x=\delta_x$ which is not an invariant measure for $x>0$.
\end{example}


\section{Proofs}
\label{sec:appendix2}

\begin{proof}[Proof of Lemma~\ref{lem:momentsPP}]
In view of~\eqref{matrixG} we obtain the vector equation
\begin{equation} \label{eq:ODE moments}
\Hcal_n(X_t) = \Hcal_n(x) + \int_0^t G_n^\top \Hcal_n(X_s)ds + M_t, \qquad t\ge0,
\end{equation}
for some local martingale $M$ with $N_n$ components. We claim that the expectation $\E[\|\Hcal_n(X_t)\|]$ is locally bounded in $t$. This follows from the inequality
\[
\E_x\big[ 1 + \|X_t\|^{2k} \big] \le \big(1+\|x\|^{2k}\big)e^{Ct}, \qquad t\ge0,
\]
which holds for some constant $C>0$ that depends on $\Gcal$ but not on $t$ or $x$. This inequality is proved using the argument in the proof of Theorem 2.10 in~\cite{Cuchiero/etal:2012}. Furthermore, in conjunction with Lemma~\ref{L:QV process} below, this also implies that $M$ is a true martingale. Taking expectations on both sides of~\eqref{eq:ODE moments} thus yields the integral equation
\[
\E[\Hcal_n(X_t)] = \Hcal_n(x) + \int_0^t G_n^\top \E[\Hcal_n(X_s)]ds, \qquad t\ge0,
\]
whose solution is $\E[\Hcal_n(X_t)] = e^{tG_n^\top} \Hcal_n(x)$. This yields~\eqref{matrixexptG}.
\end{proof}

\begin{lemma} \label{L:QV process}
Let $p\in\Pol(E)$. The local martingale $M_t=p(X_t)-\int_0^t \Gcal p(X_s)ds$ admits a predictable quadratic variation process, given by $\langle M,M\rangle_t = \int_0^t (\Gcal p^2 - 2p\Gcal p)(X_s)ds$.
\end{lemma}

\begin{proof}
Squaring the expression for $M_t$ and rearranging yields
\begin{align*}
p(X_t)^2 - M_t^2
&= 2p(X_t)\int_0^t \Gcal p(X_s)ds - \left( \int_0^t \Gcal p(X_s)ds\right)^2 \\
&= 2\left(M_t + \int_0^t \Gcal p(X_s)ds\right)\int_0^t \Gcal p(X_s)ds - \left( \int_0^t \Gcal p(X_s)ds\right)^2 \\
&= 2 M_t \int_0^t \Gcal p(X_s)ds + \left( \int_0^t \Gcal p(X_s)ds\right)^2 \\
&= 2 \int_0^t M_s \Gcal p(X_s)ds + \left( \int_0^t \Gcal p(X_s)ds\right)^2 + \text{(local martingale)} \\
&= 2 \int_0^t \left(p(X_s) - \int_0^s\Gcal p(X_u)du\right) \Gcal p(X_s)ds \\
&\qquad\qquad\qquad\qquad + \left( \int_0^t \Gcal p(X_s)ds\right)^2 + \text{(local martingale)} \\
&=2 \int_0^t p(X_s) \Gcal p(X_s)ds + \text{(local martingale)},
\end{align*}
where the last equality follows from the identity $(\int_0^t g(s)ds)^2=2\int_0^t g(s) \int_0^s g(u)du\,ds$ with $g(t)=\Gcal p(X_t)$. Therefore, since $p^2$ is also a polynomial and hence in the domain of $\Gcal$, we obtain
\[
M^2_t - \int_0^t \left( \Gcal p^2 (X_s) - 2p(X_s)\Gcal p(X_s)\right) ds = \text{(local martingale)}.
\]
This implies the assertion of the lemma.
\end{proof}

\begin{proof}[Proof of Theorem~\ref{thm:timeconsistency}]
Clearly if~\eqref{eq:timeconsistency} holds then $Y$ is an $n$-Markov cubature rule for $X$ on $\T$. Conversely, suppose that $Y$ is an $n$-Markov cubature rule for $X$ on $\T$. By an induction argument it is enough to show~\eqref{eq:timeconsistency} with $l=2$. To this end, fix $p,q\in\Pol_n(E)$ with $pq\in\Pol_n(E)$ and let $s,t\in\T$ be such that $0\leq s\leq t$. Define the function
\begin{displaymath}
\widetilde{p}(x)= \E_x[p(X_{t-s})].
\end{displaymath}
Since $X$ is a polynomial process, by Lemma~\ref{lem:momentsPP} the function $\widetilde{p}$ is a polynomial and $\widetilde{p}q\in\Pol_n(E)$. On the other hand, by the definition of a Markov cubature rule and the stability under differences of $\T$ we have
\begin{displaymath}
\begin{split}
\E_x[\widetilde{p}(X_s)q(X_s)]&=\E_x[\widetilde{p}(Y_s)q(Y_s)],\\
\widetilde{p}(x)&=\E_x[p(Y_{t-s})]
\end{split}
\end{displaymath}
for all $x\in E^y$. Therefore, as $X$ and $Y$ are Markov processes, we conclude that
\begin{align*}
\E_x[p(X_t)q(X_s)]&=\E_x[q(X_s)\E_{X_s}[p(X_{t-s})]]\\
&=\E_x[\widetilde{p}(X_s)q(X_s)]\\
&=\E_x[\widetilde{p}(Y_s)q(Y_s)]\\
&=\E_x[q(Y_s)\E_{Y_s}[p(Y_{t-s})]]\\
&=\E_x[p(Y_t)q(Y_s)]
\end{align*}
for all $x\in E^Y$.
\end{proof}

\begin{proof}[Proof of Lemma~\ref{L:rowsLaddto0}]
Denote by $v\in\R^{N_n}$ the coordinates of the constant polynomial $\mathbf{1}$ with respect to the basis $h_1(x),\ldots, h_{N_n}(x)$. We have that
\begin{displaymath}
Hv={\rm 1_{M}},
\end{displaymath}
the vector of 1's in $\R^M$. Additionally by~\eqref{matrixHG},
\begin{displaymath}
HGv=(\Gcal\mathbf{1}(x_i))_{i=1}^M=0.
\end{displaymath}		
Hence
\begin{displaymath}
L{{\rm 1}_M}=LHv=HGv=0.
\end{displaymath}
\end{proof}

\begin{proof}[Proof of Theorem~\ref{thm:continuouscubature}]
$\ref{thm:continuouscubature:1}\Rightarrow \ref{thm:continuouscubature:2}$ Let $L$ be the transition rate matrix of the $n$-Markov cubature rule $Y$. Equations~\eqref{matrixHexptG} and~\eqref{eq:Matchmoments} imply that for all $i=1,\ldots, M$, $j=1,\ldots, N_n$ and $t\geq 0$
\begin{displaymath}
(H\exp(tG))_{ij}=(\exp(tL)H)_{ij}.
\end{displaymath}
Hence, $H\exp(tG)=\exp(tL)H$ for all $t\geq 0$. Differentiating with respect to $t$ and evaluating at $t=0$ we obtain \ref{thm:continuouscubature:2}.

$\ref{thm:continuouscubature:2} \Leftrightarrow \ref{thm:continuouscubature:3}$ This follows directly from Lemma~\ref{L:rowsLaddto0}.

$\ref{thm:continuouscubature:2} \Rightarrow \ref{thm:continuouscubature:4}$ By~\eqref{matrixHG}
\begin{displaymath}
\text{the $i$-th row of $HG$}=\Gcal \Hcal_n(x_i)^\top
\end{displaymath}
for all $i=1,\ldots, M$. On the other hand, the $i$-th row of $LH$ can be written as a cone combination of the form
\begin{equation}
\label{matrixLconecomb}
\sum_{j\neq i}L_{ij}(\Hcal_n^\top(x_j)-\Hcal_n^\top(x_i)),
\end{equation}
where the coefficients $L_{ij}$ are nonnegative. Since $HG=LH$ we conclude \ref{thm:continuouscubature:4}.

$\ref{thm:continuouscubature:4}\Rightarrow \ref{thm:continuouscubature:1}$ Condition \ref{thm:continuouscubature:4} implies the existence of coefficients $L_{ij}\geq 0$ for $i\neq j$ such that~\eqref{matrixLconecomb} is equal to the $i$-th row of $HG$ for all $i$. Hence, we can find a transition rate matrix $L$ such that $HG=LH$. This implies, by an induction argument, that
\begin{displaymath}
HG^l=L^lH\text{ for all $l\in\N$.}
\end{displaymath}
This in turn implies that
\begin{equation}
\label{matrixHexptG=matrixexptLH}
H\exp(tG)=\exp(tL)H.
\end{equation}
Since $(\exp(tL))_{t\geq 0}$ defines a transition semigroup, we can define a Markov process with state space $E^Y$ by
\begin{displaymath}
\P_{x_i}^Y(Y_t=x_j)=(\exp(tL))_{ij}.
\end{displaymath}
Equations~\eqref{matrixHexptG} and~\eqref{matrixHexptG=matrixexptLH} imply that $\E_x[h_j(X_t)]=\E_x[h_j(Y_t)]$ for all $x\in E$ and $j=1,\ldots, N_n$, i.e. $Y$ defines a continuous time $n$-Markov cubature rule.

\bigskip
Suppose now that $M=N_n$ and the matrix $H$ is invertible. For all $j=1,\ldots, N_n$ define $\widetilde{h}_j$ as the polynomial whose coordinates with respect to the basis $(h_1,\ldots,h_{N_n})$ are the $j$-th column of $H^{-1}$.  We have that $\widetilde{\beta}=(\widetilde{h}_1,\ldots,\widetilde{h}_{N_n})\subset\Pol_n(E)$ is a basis that satisfies $\widetilde{h}_{j}(x_i)=\delta_{ij}$.

Given a Markov cubature rule $Y$, $Y$ is a cubature rule with respect to any basis of $\Pol_n(E)$. In particular with respect to the basis $\widetilde{\beta}$. Observe that in this case
\begin{displaymath}
\widetilde{H}=(\widetilde{h}_j(x_i))_{ij}=\Id_{N_n},
\end{displaymath}
the identity matrix. Hence, the equivalence between \ref{thm:continuouscubature:1} an \ref{thm:continuouscubature:5} follows from the equivalence between \ref{thm:continuouscubature:1} and \ref{thm:continuouscubature:3}.
\end{proof}

\begin{proof}[Proof of Lemma~\ref{lem:existenceofQ}]
Equations~\eqref{matrixHexptG},~\eqref{longrunmoments} and~\eqref{eq:longrunweights} imply that
 	\begin{displaymath}
		\lim_{t\to\infty} H{\e^{tG}}=W^\top H,
	\end{displaymath}
where $W\in\R^{M\times M}$ is the matrix with all columns equal to $w$. Since $H$ has rank $N_n$ the set
	\begin{displaymath}
		\Bcal=\{\widetilde{W}H:\widetilde{W}\in\R^{M\times M}\text{ has positive entries}\}
	\end{displaymath}
is an open set in $\R^{M\times N_n}$. Then, for $t$ large enough and defining $P(t)=e^{tG}$, we have
	\begin{equation}
	\label{eq:HP=QH}
		HP(t)=Q(t)H\in\Bcal.
	\end{equation}
The argument to show that the rows of $Q(t)$ add up to 1 is similar to that of Lemma~\ref{L:rowsLaddto0}. Suppose that $v\in\R^{N_n}$ are the coordinates of the constant polynomial $\mathbf{1}$ with respect to the basis $h_1,\ldots,h_{N_n}$. We have that ${{\rm 1}_M}=Hv,$ where ${{\rm 1}_M}\in\R^M$ is the vector of 1's. Since $\mathbf{1}$ is an eigenvalue of $\Gcal$ with corresponding eigenvalue 0, $v$ is an eigenvector of $P(t)$ with eigenvalue 1. Hence, $HP(t)v={\rm 1_M.}$  This observation together with~\eqref{eq:HP=QH} implies that $Q(t){{\rm 1}_M}=Q(t)Hv={{\rm 1}_M}$  and the rows of $Q(t)$ add up to 1, i.e. $Q(t)$ is a probability matrix.
\end{proof}

\begin{proof}[Proof of Theorem~\ref{thm:discretecubature}]
Lemma~\ref{lem:existenceofQ} guarantees that for $\Delta$ large enough, there exists a probability matrix $Q\in\R^{M\times M}$ such that
\begin{equation}
 \label{eq:HP=QH2}
 H{\e}^{\Delta G}=QH,
 \end{equation}
with $H$ defined in~\eqref{matrixH}.

Let $Y$ be the time-homogeneous Markov Process with transition probability matrix $Q$ as in~\eqref{eq:HP=QH2} and state space $E^Y=\{x_1,\ldots,x_M\}$. By~\eqref{matrixHexptG}, $Y$ is an $n$-Markov cubature for $X$ on $\{\Delta\}$. Remark~\ref{rem:timesetsums} implies that $Y$ is also an $n$-Markov cubature for $X$ on $\{l\Delta:l\in\N\}$.
\end{proof}


\bibliographystyle{plainnat}
\bibliography{bibl.bib}


\newpage

\begin{table}
\centering
\begin{tabular}{|c|c|c|c|}
 \hline
 $\exp(X_0+x)$ & $\widetilde{P}^A_x$ &  $P^{A,bin}_x$ & Rel. diff. \\
 \hline
 \hline
  80 & 21.6180 & 21.6059 & 0.0006  \\
  85 & 18.0206 & 18.0374 & 0.0009 \\
  90 & 14.9271 & 14.9187 & 0.0006 \\
  95 & 12.2445 & 12.2314 & 0.0011 \\
  100 & 9.9539 & 9.9458 &  0.0008 \\
  105 & 8.0282 & 8.0281 &  0.0000 \\
  110 & 6.4324 & 6.4352 & 0.0004 \\
  115 & 5.1267 & 5.1265 & 0.0000 \\
  120 & 4.0697 & 4.0611 & 0.0021 \\                              
  \hline
\end{tabular}
\caption{American put option prices obtained with the approximate Markov cubature method compared to prices calculated with a binomial tree approximation in a Black--Scholes model. Parameters of the model and the option: $r=0.06$, $\sigma=0.4$, $X_0=\log(K)=\log(100)$ and $T=0.5$. The first column contains the initial asset prices considered. The second column contains the prices $\widetilde{P}^A_x$  obtained with the approximate Markov cubature method with $M=40$ cubature points, $n=4$ moments and $N_{time}=1000$ time steps. The third column contains the prices $P^{A,bin}_x$ obtained with a binomial tree approximation with 1000 time steps. The last column shows the relative difference.
}\label{T:PricesBS}
\end{table}

\begin{table}
\centering
\begin{tabular}{|c|c|c|c|}
 \hline
 $x$ & $\widetilde{P}^A_x$ &  $P^{A,LS}_x$ & Rel. diff. \\
 \hline
 \hline
  0.1 & 0.5436 & 0.5434 & 0.0003   \\
  0.2 & 0.4273 & 0.4286 & 0.0030 \\
  0.3 & 0.3251 & 0.3265 & 0.0043  \\
  0.4 & 0.2468 & 0.2478 & 0.0041 \\
  0.5 & 0.1839 & 0.1852 & 0.0069 \\
  0.6 & 0.1327 & 0.1339 & 0.0087 \\
  0.7 & 0.0912 & 0.0924 & 0.0132  \\
  0.8 & 0.0580 & 0.0590 & 0.0172 \\
  0.9 & 0.0323 & 0.0329 & 0.0196 \\                                        
  \hline                          
\end{tabular}
\caption{American put option prices obtained with the approximate Markov cubature method compared to prices calculated with a Longstaff--Schwartz algorithm in a Jacobi model for exchange rates. Parameters of the model and the option:  $r=0$, $\kappa=1$, $\theta=0.5$, $x_{min}=0$, $x_{max}=1$, $X_0=0$, $K=\exp(0.5)$ and $T=0.5$. The first column contains the initial values of $x$ considered. The second column contains the prices $\widetilde{P}^A_x$  obtained with the approximate Markov cubature method with $M=40$ cubature points, $n=4$ moments and $N_{time}=1000$ time steps. The third column contains the prices $P^{A,LS}_x$ obtained with a 1000 time step Longstaff--Schwartz algorithm. The last column shows the relative difference.}\label{T:PricesJacobi}
\end{table}

\begin{figure}
\centering
\includegraphics[scale=.5]{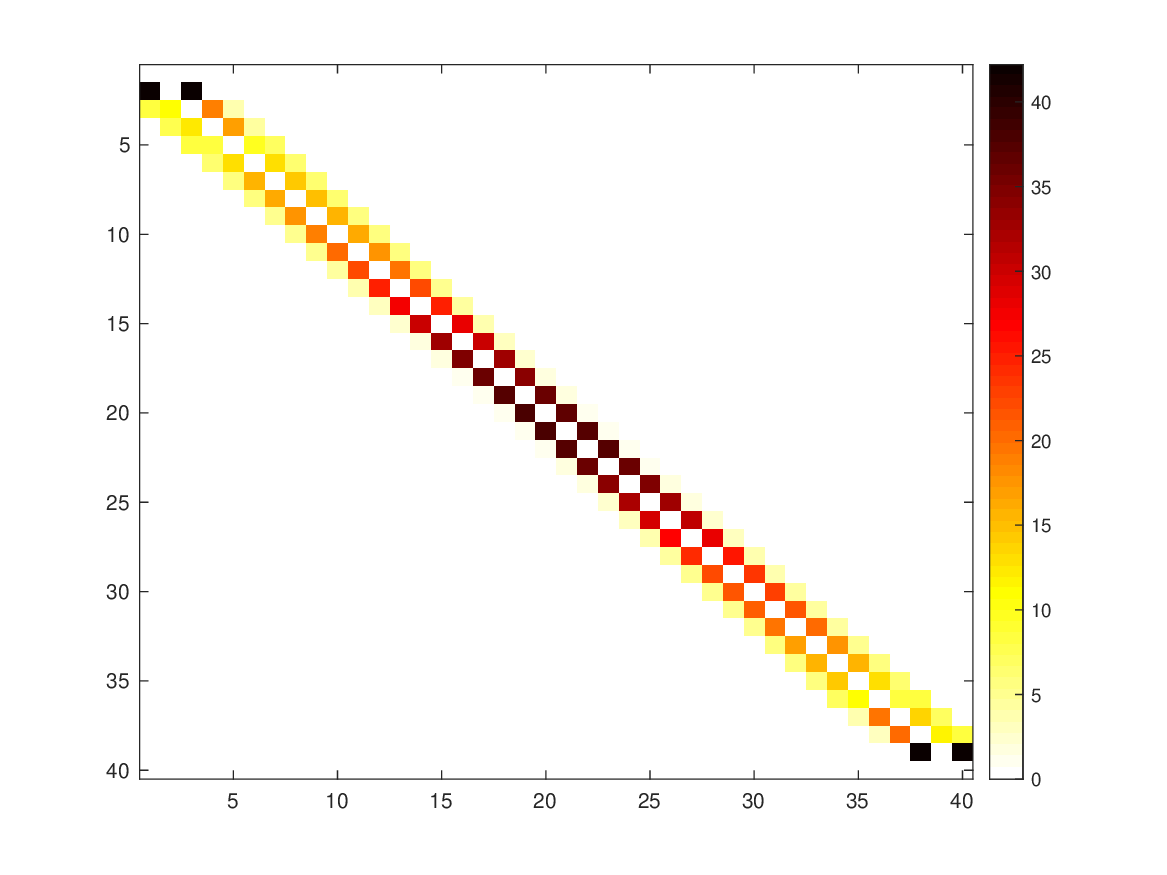}
\caption{Colormap of the off-diagonal values of the transition rate matrix $L$ obtained with the quadratic programming problem for the Black--Scholes model. Parameters:  $r=0.06$, $\sigma=0.4$, $T=0.5$, $M=40$ and $n=4$.}\label{fig:colormapBS}
\end{figure}

\begin{figure}
\centering
\includegraphics[scale=.5]{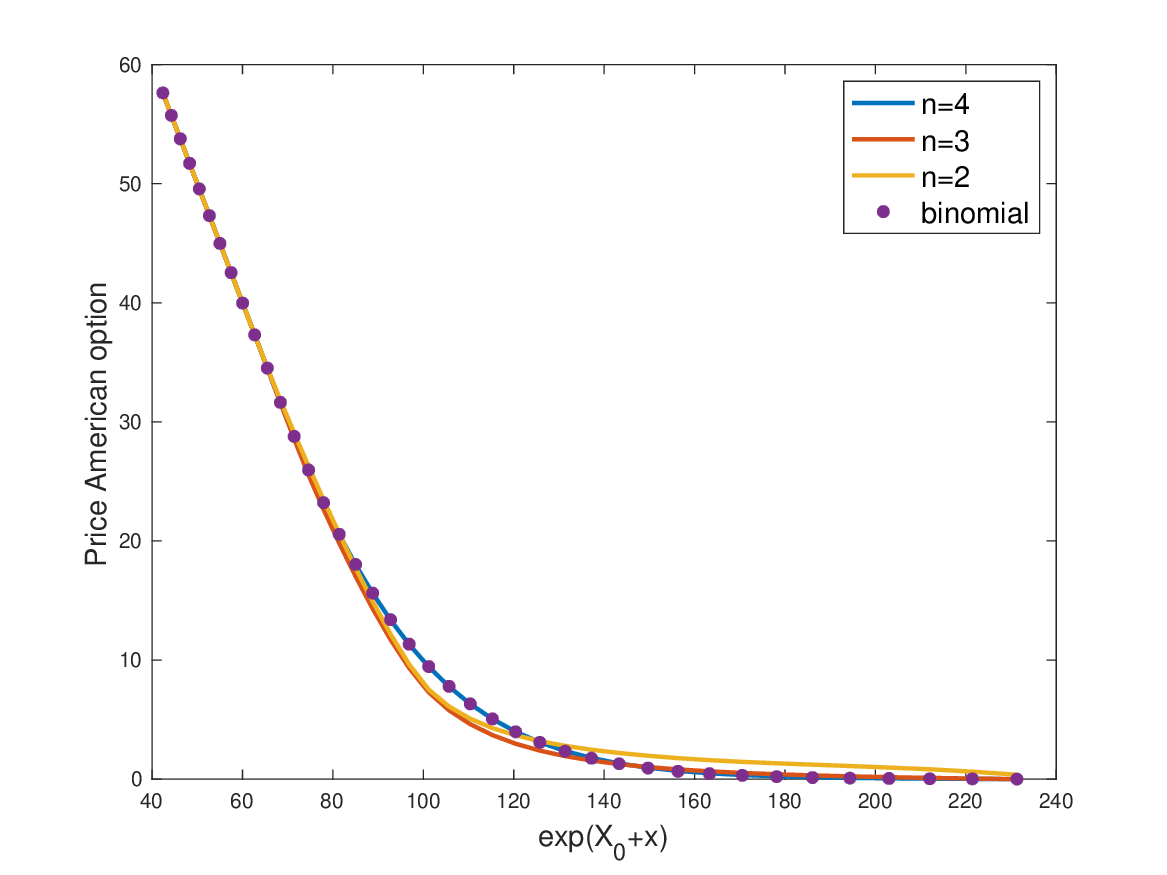}
\caption{American put option prices using the approximate Markov cubature method for different values of $n$ and the 1000--time step binomial tree approximation in a Black--Scholes model. Parameters:  $r=0.06$, $\sigma=0.4$, $X_0=\log(K)=\log(100)$, $T=0.5$, $M=40$ and $N_{time}=1000$.}\label{fig:American_BS}
\end{figure}

\begin{figure}
\centering
\includegraphics[scale=.5]{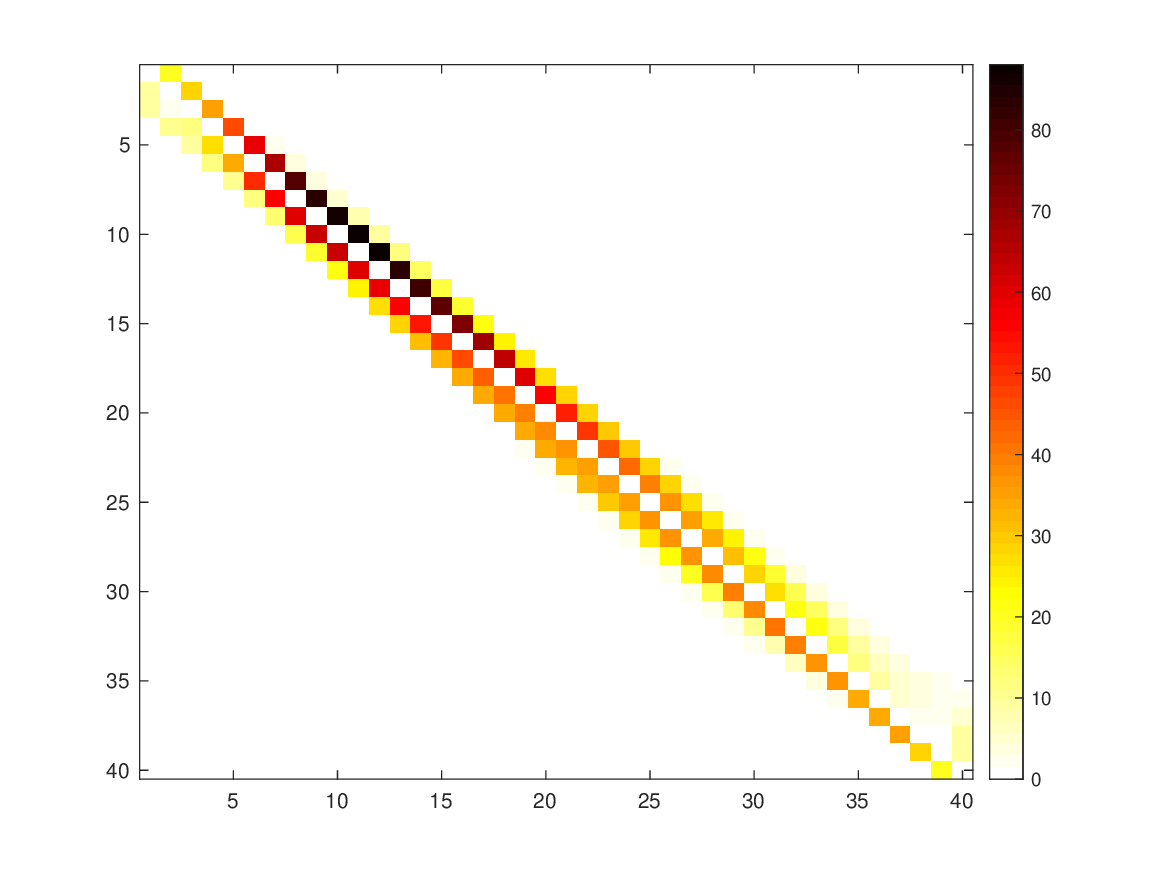}
\caption{Colormap of the off-diagonal values of the transition rate matrix $L$ obtained with the quadratic programming problem for the Jacobi model of exchange rates. Parameters:   $r=0$, $\kappa=1$, $\theta=0.5$, $x_{min}=0$, $x_{max}=1$, $M=40$ and $n=4$.}\label{fig:colormapJacobi}
\end{figure}

\begin{figure}
\centering
\includegraphics[scale=.5]{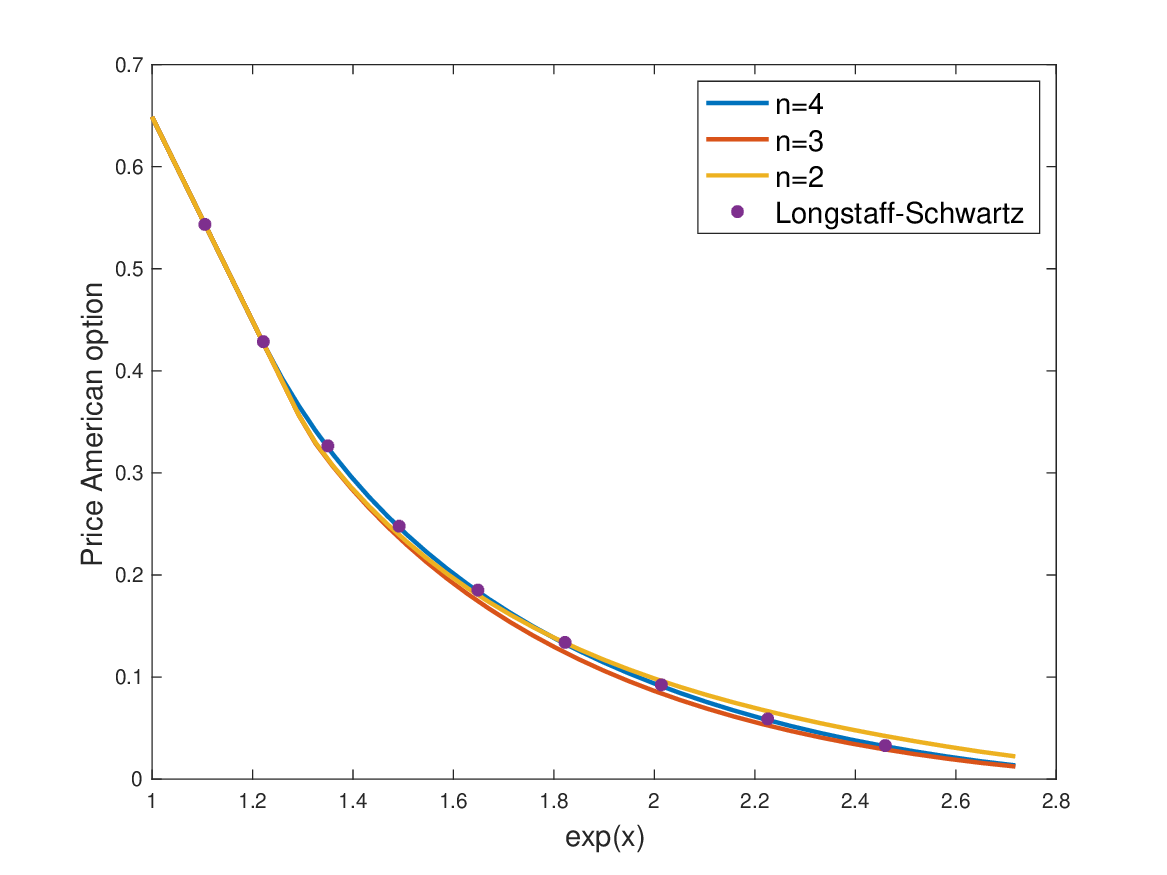}
\caption{American put option prices using the approximate Markov cubature method for different values of $n$ and the 1000--time step Longstaff--Schwartz algorithm in a Jacobi model of exchange rates. Parameters:  $r=0$, $\kappa=1$, $\theta=0.5$, $x_{min}=0$, $x_{max}=1$, $X_0=0$, $K=\exp(0.5)$, $T=0.5$, $M=40$ and $N_{time}=1000$.}\label{fig:American_Jacobi}
\end{figure}

\begin{figure}
\centering
\includegraphics[scale=.5]{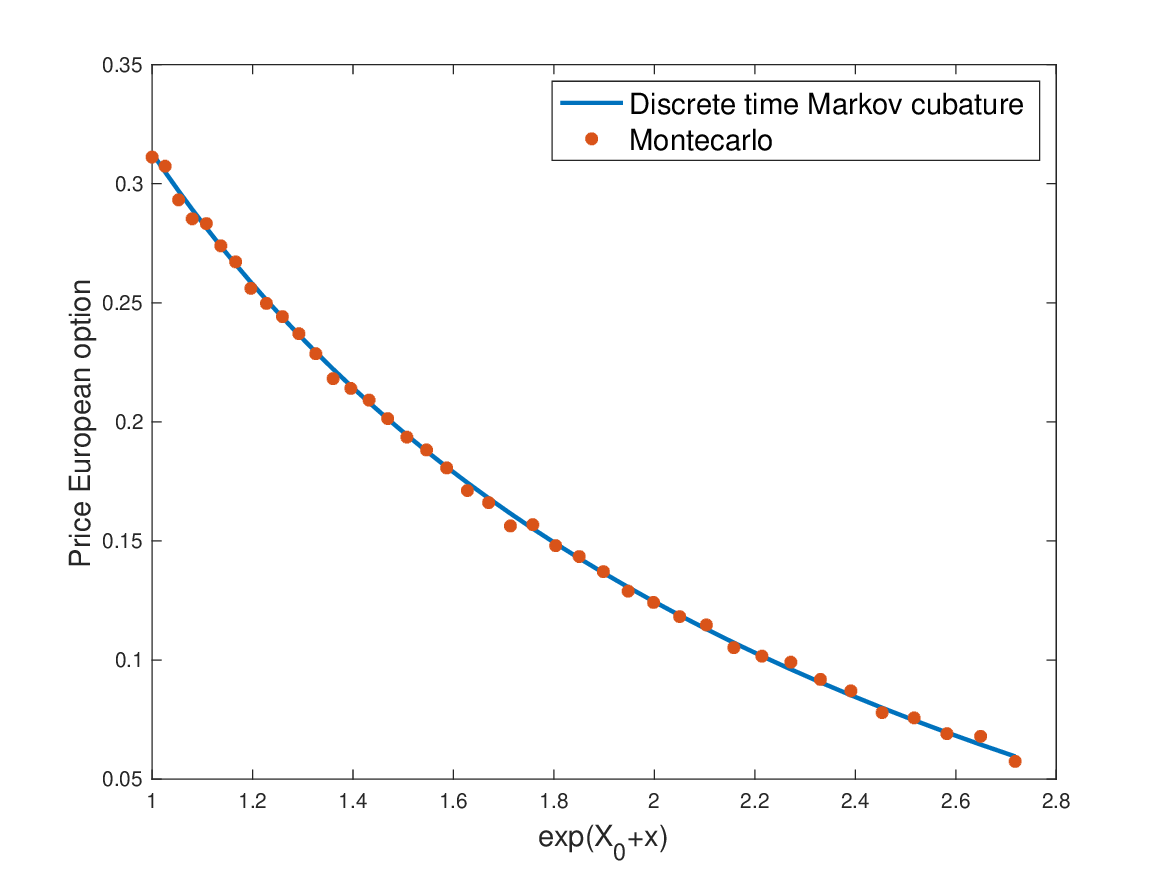}
\caption{European put option prices using the discrete time Markov cubature method and Monte-Carlo simulation in a Jacobi model of exchange rates. Parameters:  $r=0$, $\kappa=1$, $\theta=0.5$, $x_{min}=0$, $x_{max}=1$, $X_0=0$, $K=\exp(0.5)$, $T=1$, $M=40$ and $n=4$.}\label{fig:European_Jacobi}
\end{figure}

\begin{figure}
\centering
\includegraphics[scale=.5]{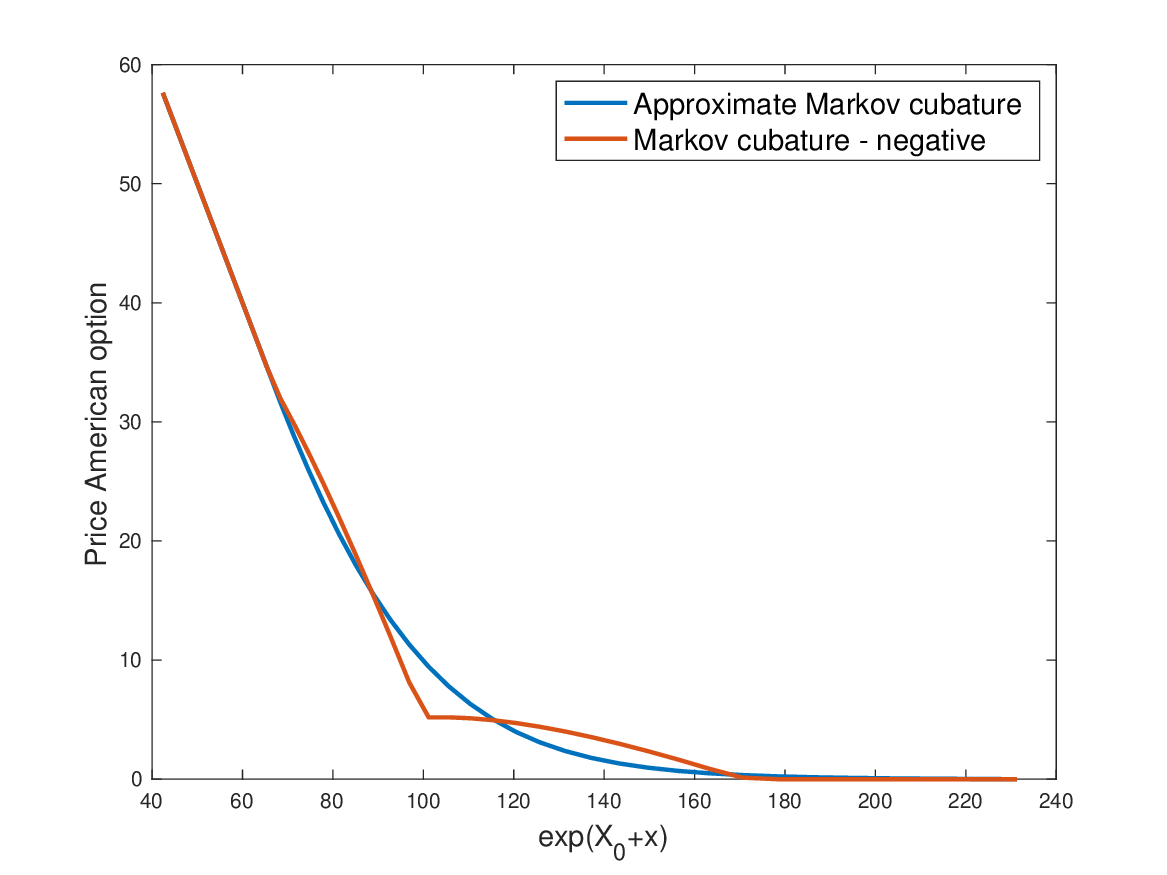}
\caption{American put option prices using negative ``probabilities'' and the approximate Markov cubature method in a Black--Scholes model. Parameters:  $r=0.06$, $\sigma=0.4$, $X_0=\log(K)=\log(100)$, $T=0.5$, $M=40$, $n=4$ and $N_{time}=1000$.}\label{fig:American_negative}
\end{figure}

\end{document}